\documentclass[10pt]{article}

\usepackage{mathtools, amsthm, accents, tikz, amssymb, latexsym, amsmath, bm, amscd, amsfonts, array, lmodern, enumerate, stmaryrd, rotating, caption, graphicx, hyperref, float,tikz-cd, color, yfonts}
\usepackage[new]{old-arrows}
\usepackage[outline]{contour}
\usetikzlibrary{calc}
\usepackage[all]{xy}
\CompileMatrices
\hypersetup{nesting=true,debug=true,naturalnames=true}

\def\TC{\protect\operatorname{TC}}

\def\cat{\protect\operatorname{cat}}
\def\zcl{\protect\operatorname{zcl}}

\def\C{\protect\operatorname{Conf}}

\def\cl{\protect\operatorname{cl}}
\def\hdim{\protect\operatorname{hdim}}

\def\bp{\mbox{\contour{black}{$\searrow$}}}
\def\sp{\mbox{\contour{black}{$\nearrow$}}}
\def\lbp{\mbox{\contour{black}{$\swarrow$}}}
\def\lsp{\mbox{\contour{black}{$\nwarrow$}}}

\newcommand{\ma}[2]{\langle\genfrac{}{}{0pt}{1}{#1}{#2}\rangle}
\newcommand{\mb}[2]{\bigl\{\hspace{-.5mm}\genfrac{}{}{0pt}{1}{#1}{#2}\hspace{-.5mm}\bigr\}}
\newcommand{\md}[4]{\left[\genfrac{}{}{0pt}{1}{#1}{#3}\genfrac{}{}{0pt}{1}{#2}{#4}\right]}
\newcommand{\nb}[6]{\langle\genfrac{}{}{0pt}{1}{#1}{#4}\genfrac{}{}{0pt}{1}{#2}{#5}\genfrac{}{}{0pt}{1}{#3}{#6}\rangle}
\newcommand{\mtt}[6]{\big[\genfrac{}{}{0pt}{1}{#1}{#4}\genfrac{}{}{0pt}{1}{#2}{#5}\genfrac{}{}{0pt}{1}{#3}{#6}\big]}
\newcommand{\mt}[6]{\left[\genfrac{}{}{0pt}{1}{#1}{#4}\genfrac{}{}{0pt}{1}{#2}{#5}\genfrac{}{}{0pt}{1}{#3}{#6}\right]}

\newcommand\xdownarrow[1][2ex]{\mathrel{\rotatebox{90}{$\xleftarrow{\rule{#1}{0pt}}$}}}

\newtheorem{proposition}{Proposition}[section]
\newtheorem{corollary}[proposition]{Corollary}
\newtheorem{definition}[proposition]{Definition}
\newtheorem{theorem}[proposition]{Theorem}
\newtheorem{remark}[proposition]{Remark}
\newtheorem{example}[proposition]{Example}
\newtheorem{lemma}[proposition]{Lemma}
\newtheorem{examples}[proposition]{Examples}

\begin{document}

\title{An algorithmic discrete gradient field and the cohomology algebra of configuration spaces of two points on complete graphs}

\author{Emilio J.~Gonz\'alez and Jes\'us Gonz\'alez}

\date{\empty}

\maketitle

\begin{abstract}
We introduce an algorithm that constructs a discrete gradient field on any simplicial complex. We show that, in all situations, the gradient field is maximal possible and, in a number of cases, optimal. We make a thorough analysis of the resulting gradient field in the case of Munkres' discrete model for $\C(K_m,2)$, the configuration space of ordered pairs of non-colliding particles on the complete graph $K_m$ on $m$ vertices. Together with the use of Forman's discrete Morse theory, this allows us to describe in full the cohomology $R$-algebra $H^*(\C(K_m,2);R)$ for any commutative unital ring $R$. As an application we prove that, although $\C(K_m,2)$ is outside the ``stable'' regime, all its topological complexities are maximal possible when $m\geq4$.
\end{abstract}



\section{Introduction}

Configuration spaces $\C(X,n)=\{(x_1,\ldots,x_n)\in X^n\colon x_i\neq x_j \text{ for } i\neq j\}$ are important ubiquitous objects in mathematics and its applications. They are reasonably understood when $X$ is a manifold of dimension at least two. For $X=\Gamma$ a graph, $\C(\Gamma,n)$ has attracted much attention in recent years due to its role in geometric group theory, and also because graph configuration spaces provide natural models for the problem of planning collision-free motion of multiple agents performing on a system of tracks, see~\cite{farbergraphs, MR1873106, MR1882808}. Yet, the current understanding of the topology of $\C(\Gamma,n)$ appears to be far more limited than that for the higher dimensional case. This is due in part to the lack of Fadell-Neuwirth fibrations relating graph configuration spaces for different values of~$n$. Informally, unlike its higher dimensional counterpart, one-dimensional motion planning actually requires global knowledge of the ambient graph. Thus, while additive information about the homology of graph configuration spaces is already available in the literature (see for instance~\cite{MR2701024, MR4080487, MR3797072, MR2745669, MR2605308, MR3000570, MR3659699, MR4050061}), explicit cup-product descriptions seem to be scarser: Farley-Sabalka's work \cite{MR2359035, MR2355034, MR2949126} (see also~\cite{tere}) relates (unordered) configurations on trees to exterior face rings, while Barnett-Farber's work~\cite{MR2491587} describes in full the rational cohomology algebra of ordered pairs of points on planar graphs. We close the gap by focusing on a diametrically different class of graphs. Indeed, we give a full description of the cohomology algebra, with any ring coefficients, of configuration spaces of ordered pairs of points on complete graphs.

Our cohomological calculations are based on discrete Morse theory techniques. For this, we describe and study an algorithm that constructs a discrete gradient field $W$ on any finite ordered abstract simplicial complex $(K,\preceq)$. The resulting field $W$ turns out to be maximal\footnote{Maximality refers to the fact that all faces and all cofaces of a $W$-critical face are involved in a Morse pairing.} and, in many cases, it is either optimal\footnote{Optimality refers to the property that the number of critical cells in a given dimension agrees with the corresponding Betti number.} (perhaps after a convenient selection of $\preceq$), or close to being so. Our algorithm can be thought of as a generalization of the inclusion-exclusion process with respect to a chosen vertex giving an optimal gradient field collapsing a full simplex to the chosen vertex. In the general case of an ordered simplicial complex $(K,\preceq)$, the ordering $\preceq$ plays a heuristic role that guides the inclusion-exclusion process.

Standard preliminary facts are reviewed in Section~\ref{secprelim}, while our algorithmic gradient field is introduced and studied in Section~\ref{algorithm-field}. Section~\ref{aplicacion} is devoted to the cohomology of configuration spaces of ordered pairs of points in complete graphs. The application to topological complexity is given in Section~\ref{TCsec}.

\section{Preliminaries}\label{secprelim}
\subsection{Munkres' model for 2-particle configuration spaces }\label{munkres-model}
Let $D$ be a full subcomplex of a given abstract simplicial complex~$X$, i.e., we assume that every simplex of $X$, whose vertices lie in $D$, is itself a simplex of $D$. Consider the (necessarily full) subcomplex $C$ of $X$ consisting of the simplices $\sigma$ of $X$ whose geometric realization $|\sigma|$ is disjoint from $|D|$. The vertices of $X$ are partitioned into those of $D$ and those of $C$ and, as observed in \cite[Lemma~70.1]{MR755006}, the linear homotopy 
$$
H\colon\left(\rule{0mm}{4mm}|X|-|D|\right)\times [0,1]\to |X|-|D|, \quad H(x,s)=(1-s)\cdot x+s\cdot \sum_{i=1}^r\frac{t_i}{\sum_{k=1}^rt_k}\rule{.3mm}{0mm}c_i,
$$
exhibits $|C|$ as a strong deformation retract of $|X|-|D|$. Here $x=\sum_{i=1}^rt_ic_i+\sum_{j=1}^\rho\tau_jd_j$ is the barycentric expression of $x\in |X|-|D|$ having $t_i>0<\tau_j$ for all $i$ and $j$, with $c_1,\ldots c_r$ vertices of $C$ ($r\geq1$) and $d_1,\ldots,d_\rho$ vertices of $D$ ($\rho\geq0$).

Let $K$ be a finite abstract \emph{ordered} simplicial complex, i.e., the vertex set $V$ of $K$ comes equipped with a partial ordering $\preceq$ which is linear upon restriction to any face. We will be interested in Munkres' model $C$ above when $X=K\times K$ is the ordered product, with $D$ corresponding to the subcomplex whose geometric realization is the diagonal $\Delta_{|K|}$ in $|K\times K|=|K|\times|K|$. The vertex set of $K\times K$ is $V\times V$, with elements denoted as columns, while a $k$-simplex of $K\times K$ is a matrix array
\begin{equation}\label{matrix-simplex}
\left[\begin{matrix}
v_{0,1} & v_{1,1} & \ldots & v_{k,1} \\
v_{0,2} & v_{1,2} & \ldots & v_{k,2}
\end{matrix}\right]
\end{equation}
of elements in $V$ satisfying:
\begin{itemize}
\item For $i=1,2$, $v_{0,i}\preceq v_{1,i}\preceq\ldots\preceq v_{k,i}$ with $\{v_{0,i}, v_{1,i},\ldots, v_{k,i}\}$ an $\ell$-face of $K$ (possibly with $\ell\leq k$).
\item For $j=0,1,\ldots,k-1$, at least one of the inequalities $v_{j,1}\preceq v_{j+1,1}$ and $v_{j,2}\preceq v_{j+1,2}$ is strict.
\end{itemize}
Such a matrix-type simplex belongs to $D$ provided its two rows are repeated: $v_{j,1}=v_{j,2}$ for $j=0,1,\ldots, k$. In particular $D$ is a full subcomplex of $K\times K$, and we get a homotopy equivalence 
\begin{equation}\label{homotopy-equivalence}
|C|\simeq \C(|K|,2).
\end{equation}
Note that a simplex~(\ref{matrix-simplex}) belongs to $C$ precisely when $v_{j,1}\neq v_{j,2}$ for $j=0,1,\ldots, k$.
In particular, the vertex set of $C$ is $V\times V\setminus\Delta_V$ (with elements denoted as column matrices).

\subsection{Discrete Morse theory}
We review the notation and facts we need from Forman's discrete Morse theory. See~\cite{MR1358614,MR1926850} for details.

Let $K$ be a finite abstract ordered simplicial complex with ordered vertex set $(V,\preceq)$. Let $(\mathcal{F},\subseteq)$ be the face poset of $K$, i.e., $\mathcal{F}$ is the set of faces of $K$ partially ordered by inclusion. For a face $\alpha\in\mathcal{F}$, we write $\alpha^{(p)}$ to indicate that $\alpha$ is $p$-dimensional, and use the notation $\alpha=[\alpha_0,\alpha_1,\cdots,\alpha_p]$, where
\begin{equation}\label{orientorder}
\alpha_0\prec\alpha_1\prec\cdots\prec\alpha_p
\end{equation}
is the ordered list of vertices of $\alpha$. We choose the orientation $\alpha$ determined by~(\ref{orientorder}). For faces $\alpha^{(p)}\subset \beta^{(p+1)}$, consider the incidence number $\iota_{\alpha,\beta}$ of $\alpha$ and $\beta$, i.e., the coefficient $\pm1$ of $\alpha$ in the expression of $\partial(\beta)$. Here $\partial$ stands for the boundary operator in the oriented simplicial chain complex $C_*(K)$, i.e.,
$$
\partial\left([v_0,v_1,\ldots,v_i]\right)=\sum_{0\leq j\leq i}(-1)^{j}[v_0,\ldots,\widehat{\hspace{.2mm}v_j\hspace{.2mm}},\ldots,v_i],
$$
where the notation $\widehat{v}$ means that vertex $v$ is to be omitted.

Think of the Hasse diagram $H_\mathcal{F}$ of $\mathcal{F}$ as a directed graph; the vertex set of $H_{\mathcal{F}}$ is~$\mathcal{F}$ and the directed edges are the ordered pairs $(\alpha^{(p+1)},\beta^{(p)})$ with $\beta\subset\alpha$. Such a directed edge will be denoted as $\alpha^{(p+1)}\searrow\beta^{(p)}$. Let $W$ be a partial matching on $H_\mathcal{F}$, i.e., a directed subgraph of $H_\mathcal{F}$ whose vertices have degree~1. Note that the vertex set of $W$ may be a proper subset of $\mathcal{F}$. The modified Hasse diagram $H_{\mathcal{F},W}$ is the directed graph obtained from $H_\mathcal{F}$ by reversing the orientation of all edges of $W$. A reversed edge is denoted as $\beta^{(p)}\nearrow\alpha^{(p+1)}$, in which case $\alpha$ is said to be collapsible and $\beta$ is said to be redundant. In this setting, a path is an alternate chain of up-going and down-going directed edges of $H_{\mathcal{F},W}$ of either of the two forms
\begin{equation}\label{path}
\alpha_0\nearrow\beta_1\searrow\alpha_1\nearrow\cdots\nearrow \beta_k\searrow\alpha_k \mbox{ \ \ or \ \ } \gamma_0\searrow\delta_1\nearrow\gamma_1\searrow\cdots\searrow\delta_k\nearrow\gamma_k.
\end{equation}
A path as the one on the left (right) hand-side of~(\ref{path}) is called an upper (respectively, lower) path, and the path is called elementary when $k=1$, or constant when $k=0$. The sets of upper and lower paths that start on a $p$-cell $\alpha$ and end on a $p$-cell $\beta$ are denoted by $\overline{\Gamma}(\alpha,\beta)$ and $\underline{\Gamma}(\alpha,\beta)$, respectively. A \emph{mixed} path $\widetilde{\lambda}$ from a face $\beta^{(p+1)}$ to a face $\alpha^{(p)}$ is the concatenation of an edge $\beta\searrow\gamma$ and an upper path $\lambda\in\overline{\Gamma}(\gamma,\alpha)$.

Concatenation of upper/lower paths yields product maps
\begin{equation}\label{concaprod}
\overline{\Gamma}(\alpha,\beta)\times\overline{\Gamma}(\beta,\gamma)\to\overline{\Gamma}(\alpha,\gamma) \mbox{ \ and \ \ }\underline{\Gamma}(\alpha,\beta)\times\underline{\Gamma}(\beta,\gamma)\to\underline{\Gamma}(\alpha,\gamma).
\end{equation}
For instance, any upper/lower path is a product of corresponding elementary paths. The multiplicity of an upper/lower path is defined, in the elementary case, as 
$$
\mu(\alpha_0\nearrow\beta_1\searrow\alpha_1)=-\iota_{\alpha_0,\beta_1}\cdot\iota_{\alpha_1,\beta_1}\mbox{ \ \ and \ \ } \mu(\gamma_0\searrow\delta_1\nearrow\gamma_1)=-\iota_{\delta_1,\gamma_0}\cdot\iota_{\delta_1,\gamma_1}
$$
and, in the general case, multiplicity is defined to be a multiplicative function with respect to~(\ref{concaprod}). In particular, the multiplicity of any constant path is~1. Likewise, $\mu(\widetilde{\lambda}):=\iota_{\gamma,\beta}\cdot\mu(\lambda)$ defines the multiplicity of the mixed path $\widetilde{\lambda}$ given by the concatenation of the edge $\beta\searrow\gamma$ and the upper path $\lambda\in\overline{\Gamma}(\gamma,\alpha)$.

A non-constant path is called a cycle if $\alpha_0=\alpha_k$, in the upper case of (\ref{path}), or $\gamma_0=\gamma_k$, in the lower case. By construction, the cycle condition can only hold with $k>1$. The matching $W$ is said to be a gradient field on~$K$ if $H_{\mathcal{F},W}$ has no cycles. In such a case, paths are referred as gradient paths, while cells of $K$ that are neither redundant nor collapsible are said to be critical.

Critical faces and gradient paths can be used to recover (co)homological information of $K$. Explicitly, the Morse chain complex $(\mu_*(K),\partial)$ is $R$-free\footnote{Coefficients are taken in a commutative unital ring $R$, as we will ultimately be interested in cup-products.} with basis in dimension $p\geq0$ given by the oriented critical faces $\alpha^{(p)}$ of~$K$, and with Morse boundary map $\partial\colon\mu_*(K)\to\mu_{*-1}(K)$ given at a critical face $\alpha^{(p)}$ by
\begin{equation}\label{morseboundary}
\partial(\alpha^{(p)})=\sum_{\beta^{(p-1)}} \left(\sum_{\widetilde{\lambda}}\mu(\widetilde{\lambda})\right)\cdot \beta,
\end{equation}
where the outer summation runs over all critical faces $\beta^{(p-1)}$, and the inner summation runs over all mixed gradient paths $\widetilde{\lambda}$ from $\alpha$ to $\beta$. The Morse cochain complex $(\mu^*(K),\delta)$ is the dual\footnote{For the sake of brevity, we will consistently omit writing asterisks for dualized objects.} of $(\mu_*(K),\partial)$. Thus, $\mu^p(K)$ is $R$-free with basis given by the duals of the oriented critical faces $\alpha^{(p)}$ of~$K$. The value of the Morse coboundary map $\delta\colon\mu^*(K)\to\mu^{*+1}(K)$ at a (dualized) critical face $\alpha^{(p)}$ is
\begin{equation}\label{morsecoboundary}
\delta(\alpha^{(p)})=\sum_{\beta^{(p+1)}} \left(\sum_{\widetilde{\lambda}}\mu(\widetilde{\lambda})\right)\cdot \beta,
\end{equation}
where the outer summation runs over all (dualized) critical faces $\beta^{(p+1)}$, and the inner summation runs over all mixed gradient paths $\widetilde{\lambda}$ from $\beta$ to $\alpha$. In other words, the Morse theoretic incidence number of critical faces $\gamma_1^{(p)}$ and $\gamma_2^{(p+1)}$ is the multiplicity-counted number of mixed gradient paths from $\gamma_2$ to $\gamma_1$.

Gradient paths yield, in addition, a homotopy equivalence between the Morse cochain complex $\mu^*(K)$ and the simplicial cochain complex $C^*(K)$. Indeed, the formul\ae
\begin{equation}\label{quasi-isomorphisms}
\begin{matrix}
\overline{\Phi}(\alpha^{(p)})=\sum_{\beta^{(p)}}\left(\sum_{\lambda\in\overline{\Gamma}(\beta,\alpha)}\mu(\lambda)\right)\beta, & (\alpha \text{ critical, } \beta \text{ arbitrary}), \\
\underline{\Phi}(\beta^{(p)})=\sum_{\alpha^{(p)}}\left(\sum_{\lambda\in\underline{\Gamma}(\alpha,\beta)}\mu(\lambda)\right)\alpha, & (\beta \text{ arbitrary, } \alpha \text{ critical}) 
\end{matrix}
\end{equation}
determine cochain maps $\overline{\Phi}\colon \mu^*(K)\to C^*(K)$ and $\underline{\Phi}\colon C^*(K)\to \mu^*(K)$ inducing cohomology isomorphisms $\overline{\Phi}^*$ and $\underline{\Phi}^*$ with $(\underline{\Phi}^*)^{-1}=\overline{\Phi}^*$. In particular, cup-products can be evaluated directly at the level of the Morse cochain complex $\mu^*(K)$. Indeed, given Morse cocycles $x,y\in\mu^*(K)$ representing respective cohomology classes $x',y'\in H^*(\mu^*(K))$, the cohomology cup product $x'\cdot y'$ is represented by the Morse cocycle
\begin{equation}\label{cupprod}
x\stackrel{\mu}{\smile}y:=\underline{\Phi}\left(\overline{\Phi}(x)\smile\overline{\Phi}(y)\right)\in\mu^*(K),
\end{equation}
where $\smile$ stands for the simplicial cup product.

\section{Algorithmic gradient field}\label{algorithm-field}
Let $K$ be a finite abstract ordered simplicial complex of dimension $d$ with ordered vertex set $(V,\preceq)$. Recall that the partial order $\preceq$ is required to restrict to a linear order on simplices of $K$. In this section, we describe and study an algorithm $\mathcal{A}$ that constructs a discrete gradient field $W$ (depends on~$\preceq$) on $K$. 

By the order-extension principle, we may as well assume $\preceq$ is linear from the outset. Let $\mathcal{F}^i$ denote the set of $i$-dimensional faces of $K$. Recall that a face $\alpha^{(i)}\in\mathcal{F}^i$ is identified with the ordered tuple $[\alpha_0,\alpha_1,\cdots,\alpha_i]$, $\alpha_0\prec\alpha_1\prec\cdots\prec\alpha_i$, of its vertices. In such a setting, we say that $\alpha_r$ appears in position~$r$ of $\alpha$. The ordered-tuple notation allows us to lexicographically extend $\preceq$ to a linear order (also denoted by $\preceq$) on the set $\mathcal{F}$ of faces of~$K$. We write $\prec$ for the strict version of $\preceq$.

For a vertex $v\in V$, a face $\alpha\in\mathcal{F}^i$ and an integer $r\geq0$, let 
$$\iota_r(v,\alpha)=\begin{cases}
\alpha\cup\{v\},&\mbox{if $\alpha\cup\{v\}\in\mathcal{F}^{i+1}$, with $v$ appearing in position $r$ of $\alpha\cup\{v\};$}\\\varnothing,& \mbox{otherwise.}
\end{cases}$$

\subsection{Acyclicity}
At the start of the algorithm we set $W:=\varnothing$ and initialize auxiliary variables $F^i:=\mathcal{F}^i$ for $0\leq i\leq d$ which, at any moment of the algorithm, keep track of $i$-dimensional faces not taking part of a pairing in~$W$. Throughout the algorithm $\mathcal{A}$, pairings $(\alpha,\beta)\in\mathcal{F}^i\times\mathcal{F}^{i+1}$ are added to $W$ by means of a family of processes $\mathcal{P}^i$ running for $i=d-1,d-2,\ldots,1,0$ (in that order), where $\mathcal{P}^i$ is executed provided (at the relevant moment) both $F^i$ and $F^{i+1}$ are not empty (so there is a chance to add new pairings to~$W$). Process $\mathcal{P}^i$ consists of three levels of nested subprocesses:
\begin{enumerate}
\item At the most external level, $\mathcal{P}^i$ consists of a family of processes $\mathcal{P}^{i,r}$ for $i+1\geq r\geq0$, executed in descending order with respect to $r$.
\item In turn, each $\mathcal{P}^{i,r}$ consists of a family of subprocesses $\mathcal{P}^{i,r,v}$ for $v\in V$, executed from the $\preceq$-largest vertex to the smallest one.
\item At the most inner level, each process $\mathcal{P}^{i,r,v}$ consists of a family of instructions $\mathcal{P}^{i,r,v,\alpha}$ for $\alpha\in\mathcal{F}^i$, executed following the $\preceq$-lexicographic order.
\end{enumerate}
Instruction $\mathcal{P}^{i,r,v,\alpha}$ checks whether, at the moment of its execution, $(\alpha,\iota_r(v,\alpha))\in F^i\times F^{i+1}$, i.e., whether $(\alpha,\iota_r(v,\alpha))$ is ``available'' as a new pairing. If so, the pairing $\alpha\nearrow\iota_r(v,\alpha)$ is added to $W$, while $\alpha$ and $\iota_r(v,\alpha)$ are removed from $F^i$ and $F^{i+1}$, respectively. By construction, at the end of the algorithm, the resulting family of pairs $W$ is a partial matching in $\mathcal{F}$. Furthermore, from its construction, 
\begin{equation}\label{maxim}
\mbox{\emph{all faces and cofaces of an unpaired cell are involved in a $W$-paring,}}
\end{equation}
so that $W$ is maximal. Most importantly:

\begin{proposition}\label{acyclicmatching}
$W$ is a gradient field.
\end{proposition}

In preparation for the proof of Proposition~\ref{acyclicmatching}, we need:
\begin{definition}\label{parespotenciales}
Let $W_{i,r,v}$ denote the collection of pairings $\alpha\nearrow\beta$ in $W$ constructed during the process $\mathcal{P}^{i,r,v}$. Consider also the collection $P_{i,r,v}$ of pairs $(\alpha,\beta)\in\mathcal{F}^i\times\mathcal{F}^{i+1}$ such that $\beta\setminus\alpha=\{v\}$ with $v$ appearing in position $r$ of $\beta$. Thus $W_{i,r,v}=P_{i,r,v}\cap W$.
\end{definition}

We start by proving that, at the moment that $\mathcal{A}$ constructs a pairing $\alpha\nearrow\beta$, $\alpha$ is in fact the smallest (with respect to $\preceq$) of the facets of $\beta$ that remain unpaired.

\begin{lemma}\label{apareaconlamenor}
Let $\alpha=[\alpha_0,\ldots,\alpha_{r},\alpha_{r+1},\ldots,\alpha_{i}]\nearrow\beta=[\alpha_0,\ldots,\alpha_{r},\beta_0,\alpha_{r+1},\ldots,\alpha_{i}]$ be a pairing in $W_{i,r+1,\beta_0}$ and let $\gamma$ be a face of $\beta$ with $\gamma=[\alpha_0,\ldots,\alpha_{r},\beta_0,\alpha_{r+1},\ldots,\widehat{\alpha_j},\ldots,\alpha_{i}]$ for $r+1\leq j\leq i$, i.e., $\gamma\prec\alpha$. Then there is an integer $\ell\in\{j+1,j+2,\ldots,i+1\}$ and a vertex $\delta_0$ with $\alpha_j\prec\delta_0$ so that $$\gamma\nearrow\delta:=[\alpha_0,\ldots,\alpha_{r},\beta_0,\alpha_{r+1},\ldots,\alpha_{j-1},\widehat{\alpha_j},\ldots,\delta_0,\ldots]$$ lies in $W_{i,\ell,\delta_0}$. In particular, the pairing $\gamma\nearrow\delta$ is constructed by $\mathcal{A}$ before the pairing $\alpha\nearrow\beta$.
\end{lemma}
\begin{proof}
Previous to the instruction $\mathcal{P}^{i,r+1,\beta_0,\alpha}$ that constructs $\alpha\nearrow\beta$, the algorithm $\mathcal{A}$ executes the instruction $\mathcal{P}^{i,j+1,\alpha_j,\gamma}$ that evaluates the potential pair $(\gamma,\beta)\in P_{i,j+1,\alpha_j}$. The latter is not an element of $W$, as $\beta$ reaches a later stage in $\mathcal{A}$. So $\gamma$ must be paired by an instruction $\mathcal{P}^{i,\ell,\delta_0,\gamma}$ previous to $\mathcal{P}^{i,j+1,\alpha_j,\gamma}$, which forces the conclusion.
\end{proof}

\begin{proof}[Proof of Proposition~\ref{acyclicmatching}]
Assume for a contradiction that there is a $W$-cycle
\begin{equation}\label{cicloelegido}
\alpha^0\nearrow\beta^0\searrow\alpha^1\nearrow\beta^1\searrow\alpha^2\nearrow\cdots\nearrow\beta^n\searrow\alpha^{n+1}=\alpha^0
\end{equation}
(the condition $n\geq1$ is forced by the definition of a gradient path). Without loss of generality, we can assume that $\alpha^0\nearrow\beta^0$ is constructed by $\mathcal{A}$ before any other pairing $\alpha^j\nearrow\beta^j$ with $1\leq j\leq n$. So, Lemma~\ref{apareaconlamenor} forces the start of the cycle to have the form
\begin{align*}
\alpha^0=[\hspace{.2mm}\alpha^0_0,.\dots\dots\dots\dots\dots\dots,&\hspace{.4mm}\alpha^0_{j_0+1},\ldots,\alpha^0_k],\\
\beta^0=[\alpha^0_0,\ldots\dots\dots\ldots\ldots,\alpha^0_{j_0},&\,\beta^0_0,\alpha^0_{j_0+1},\ldots,\alpha^0_k],\\
\alpha^1=[{\alpha^0_0,\hspace{.1mm}.\ldots,\widehat{\alpha^0_\ell},\ldots,\hspace{.1mm}\alpha^0_{j_0}},\beta_0^0,&\hspace{.4mm}\alpha^0_{j_0+1},\ldots,\alpha^0_k].
\end{align*}
Assume inductively $\alpha^j=[\,\cdots,\beta^0_0,\alpha^0_{j_0+1},\ldots,\alpha^0_k]$ with $\beta_0^0$ appearing in position $j_0$ (so $\alpha^j\neq\alpha^0$). The choosing of $\alpha^0\nearrow\beta^0$ implies that $\beta^j$ is obtained from $\alpha^j$ by inserting a vertex $v$ on the left of $\beta^0_0$ (i.e., $v<\beta^0_0$). A new application of Lemma~\ref{apareaconlamenor} (together with the choosing of $\alpha^0\nearrow\beta^0$) then shows that $\alpha^{j+1}$ must be obtained from $\beta^j$ by removing a vertex other than $\beta^0_0,\alpha^0_{j_0+1},\ldots,\alpha_k^0$. Thus $\alpha^{j+1}=[\cdots,\beta_0^0,\alpha^0_{j_0+1},\ldots,\alpha^0_k]$, which is again different from $\alpha_0$. Iterating, we get a situation incompatible with the equality in~(\ref{cicloelegido}).
\end{proof}

We have noted that, when $K$ is a full simplex, $\mathcal{A}$ constructs the standard (and optimal) gradient field determined by inclusion-exclusion of a fixed vertex (the largest one in the selected order $\preceq$). Optimality is reached in many other standard situations.

\begin{examples}\label{projectiveandtorus}{\em
Figure~\ref{RP2} gives a triangulation of the projective plane $\mathbb{R}P^2$. The gradient field shown by the heavy arrows is determined by $\mathcal{A}$ using the indicated ordering of vertices. The only critical faces are $[6]$ (in dimension 0), $[2,5]$ (in dimension~1) and $[1,3,4]$ (in dimension 2), so optimality of the field follows from the known mod-2 homology of $\mathbb{R}P^2$. Although the gradient field depends on the ordering of vertices, we have verified with the help of a computer that, in this case, all possible 720 gradient fields (coming from the corresponding $6!$ possible orderings of vertices) are optimal. A corresponding optimal gradient field on the 2-torus (and the vertex-order rendering it) is shown in Figure~\ref{gradienttorus}. This time the critical faces are $[9]$ (in dimension 0), $[2,8]$ and $[5,8]$ (in dimension 1) and $[1,3,7]$ (in dimension 2). The torus case is interesting in that there are vertex orderings that yield non-optimal gradient fields. In general, a plausible strategy for choosing a convenient ordering of vertices consists on assuring the largest possible number of vertices with high $\preceq$-tag so that no two such vertices lie on a common face. For instance, in our torus example, no pair of vertices taken from 7, 8 and 9 lie on a single face. 
}\end{examples}

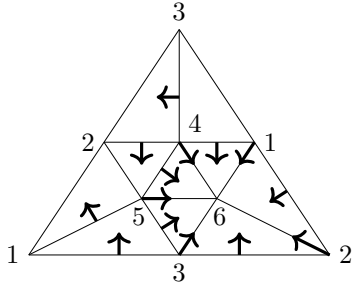
\begin{figure}
\centering
\begin{tikzpicture}[x=1cm,y=1cm]
\node[left] at (0,0) {1};\node[left] at (1,1.5) {2};\node[above] at (2,3) {3};
\node[right] at (3,1.5) {1};\node[right] at (4,0) {2};\node[below] at (2,0) {3};
\node[right] at (2,1.75) {4};\node[below] at (1.46,.75) {5};\node[below] at (2.54,.75) {6};
\draw[very thin](0,0)--(4,0);\draw[very thin](0,0)--(2,3);\draw[very thin](4,0)--(2,3);
\draw[very thin](1,1.5)--(3,1.5);\draw[very thin](1,1.5)--(2,0);\draw[very thin](3,1.5)--(2,0);
\draw[very thin](2,1.5)--(2,3);
\draw[very thin](0,0)--(1.5,.75);
\draw[very thin](4,0)--(2.5,.75);
\draw[very thin](2,1.5)--(1.5,.75);
\draw[very thin](2,1.5)--(2.5,.75);
\draw[very thin](1.5,.75)--(2.5,.75);
\draw[->,very thick](1.2,0)--(1.2,.3);
\draw[->,very thick](2.8,0)--(2.8,.3);
\draw[->,very thick](.9,.45)--(.75,.7);
\draw[->,very thick](2,0)--(2.2,.3);
\draw[->,very thick](4,0)--(3.5,.25);
\draw[->,very thick](3,1.5)--(2.8,1.2);
\draw[->,very thick](2,1.5)--(2.2,1.2);
\draw[->,very thick](1.5,.75)--(1.9,.75);
\draw[->,very thick](3.43,.85)--(3.2,.7);
\draw[->,very thick](2.5,1.5)--(2.5,1.2);
\draw[->,very thick](1.5,1.5)--(1.5,1.2);
\draw[->,very thick](2,2.1)--(1.7,2.1);
\draw[->,very thick](1.76,.35)--(2,.5);
\draw[->,very thick](1.76,1.15)--(2,.97);
\end{tikzpicture}
\caption{Algorithmic gradient field in the projective plane}
\label{RP2}
\end{figure}

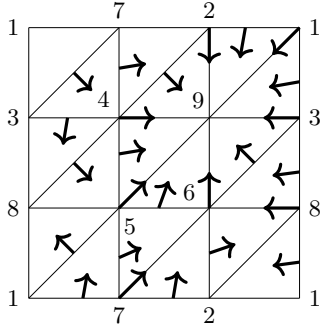
\begin{figure}
\centering
\begin{tikzpicture}[x=1.2cm,y=1.2cm]
\draw[very thin](0,0)--(3,0);\draw[very thin](0,0)--(0,3);
\draw[very thin](0,1)--(3,1);\draw[very thin](1,0)--(1,3);
\draw[very thin](0,2)--(3,2);\draw[very thin](2,0)--(2,3);
\draw[very thin](0,3)--(3,3);\draw[very thin](3,0)--(3,3);
\draw[very thin](0,0)--(3,3);\draw[very thin](0,1)--(2,3);
\draw[very thin](0,2)--(1,3);\draw[very thin](1,0)--(3,2);
\draw[very thin](2,0)--(3,1);
\node[left] at (0,0) {\small 1};\node[left] at (0,3) {\small 1};
\node[right] at (3,0) {\small 1};\node[right] at (3,3) {\small 1};
\node[above] at (2,3) {\small 2};\node[below] at (2,0) {\small 2};
\node[above] at (1,3) {\small 7};\node[below] at (1,0) {\small 7};
\node[left] at (0,2) {\small 3};\node[right] at (3,2) {\small 3};
\node[left] at (0,1) {\small 8};\node[right] at (3,1) {\small 8};
\node[right] at (.95,.8) {\small 5};
\node[left] at (1.95,1.17) {\small 6};
\node[left] at (2.05,2.19) {\small 9};
\node[left] at (1,2.2) {\small 4};
\draw[->,very thick](3,3)--(2.7,2.7);
\draw[->,very thick](2,3)--(2,2.6);
\draw[->,very thick](3,2)--(2.6,2);
\draw[->,very thick](1,2)--(1.4,2);
\draw[->,very thick](1,1)--(1.3,1.3);
\draw[->,very thick](2,1)--(2,1.4);
\draw[->,very thick](1,0)--(1.3,.3);
\draw[->,very thick](3,1)--(2.6,1);
\draw[->,very thick](2.4,3)--(2.34,2.66);
\draw[->,very thick](3,2.4)--(2.66,2.34);
\draw[->,very thick](.5,.5)--(.3,.7);
\draw[->,very thick](.6,0)--(.65,.28);
\draw[->,very thick](3,.4)--(2.7,.36);
\draw[->,very thick](1.5,2.5)--(1.7,2.3);
\draw[->,very thick](2,.5)--(2.29,.6);
\draw[->,very thick](1.6,0)--(1.65,.3);
\draw[->,very thick](.43,2)--(.38,1.7);
\draw[->,very thick](2.5,1.5)--(2.3,1.7);
\draw[->,very thick](.5,2.5)--(.7,2.3);
\draw[->,very thick](3,1.4)--(2.7,1.36);
\draw[->,very thick](1,1.6)--(1.3,1.65);
\draw[->,very thick](1,2.55)--(1.3,2.6);
\draw[->,very thick](.5,1.5)--(.7,1.3);
\draw[->,very thick](1.44,1)--(1.55,1.3);
\draw[->,very thick](1,.45)--(1.25,.56);

\end{tikzpicture}
\caption{Algorithmic gradient field in the 2-torus}
\label{gradienttorus}
\end{figure}

The option $\alpha\prec\gamma$ ruled out by the hypotheses in Lemma~\ref{apareaconlamenor} is addressed in: 
\begin{lemma}\label{apareaconlamayor}
Let $\alpha=[\alpha_0,\ldots,\alpha_i,\alpha_{i+1},\ldots,\alpha_k]\nearrow\beta=[\alpha_0,\ldots,\alpha_i,\beta_0,\alpha_{i+1},\ldots,\alpha_k]$ lie in $W_{k,i+1,\beta_0}$ and let $\gamma$ be a face of $\beta$ with $\alpha\prec\gamma$, i.e., $\gamma=[\alpha_0,\ldots,\widehat{\alpha_j},\ldots,\alpha_i,\beta_0,\alpha_{i+1},\ldots,\alpha_k]$ for $0\leq j\leq i$. Assume $\gamma\nearrow\delta$ is a pairing constructed \emph{after} the pairing $\alpha\nearrow\beta$. Then $\delta$ is obtained from~$\gamma$ by inserting a vertex $\delta_0$ which is $\prec$-smaller than $\beta_0$, i.e., $\delta=(\ldots,\delta_0,\ldots,\beta_0,\alpha_{i+1},\ldots,\alpha_k)$.
\end{lemma}
\begin{proof}
The assertion follows from the definition of the algorithm $\mathcal{A}$ noticing that $\alpha_{i+1}$ appears in position $i+1$ in $\gamma$.
\end{proof}

\subsection{Gradient field via a faster algorithm}
The proof of Proposition~\ref{acyclicmatching} makes critical use of ``timing'' in the construction of $W$-pairs whithin the algorithm~$\mathcal{A}$. Such a characteristic will be modified next in order to get a more efficient and faster version of $\mathcal{A}$. While timing of $W$-pairs construction will be altered, we shall show that the new algorithm constructs the same gradient field.

The algorithm $\overline{\mathcal{A}}$ in this subsection, initialized with auxiliary variables $\overline{W}$ and $\overline{F}^i$ analogous to those for its counterpart $\mathcal{A}$, consists of a family of processes $\overline{\mathcal{P}}^i$ running for $i=d-1,d-2,\ldots,1,0$ (in that order). Each $\overline{\mathcal{P}}^i$ is executed under the same conditions (with respect to $\overline{F}^i$ and $\overline{F}^{i+1}$) as its analogue $\mathcal{P}^i$, but consists only of two (rather than three) levels of nested subprocess. Namely, at the most external level, $\overline{\mathcal{P}}^i$ consists of a family of processes $\overline{\mathcal{P}}^{i,v}$ for $v\in V$, executed from the $\preceq$-largest vertex to the smallest one. In turn, each process $\overline{\mathcal{P}}^{i,v}$ consists of a family of instructions $\overline{\mathcal{P}}^{i,v,\alpha}$ for $\alpha\in\mathcal{F}^i$, executed following the $\preceq$-lexicographic order. Instruction $\overline{\mathcal{P}}^{i,v,\alpha}$ checks whether, at that moment, $(\alpha,\{v\}\cup\alpha))\in\overline{F}^i\times\overline{F}^{i+1}$ (i.e., availability). If so, the pairing $\alpha\nearrow \{v\}\cup\alpha$ is added to $\overline{W}$, while $\alpha$ and $\{v\}\cup\alpha$ are removed from $\overline{F}^i$ and $\overline{F}^{i+1}$, respectively. Thus, the difference with the algorithm $\mathcal{A}$ is that, in order to construct a pairing $\alpha\nearrow \{v\}\cup\alpha$ in $\overline{W}$, we do not care about the position of $v$ in $\{v\}\cup\alpha$. As we will explain next, such a situation means that algorithm~$\overline{\mathcal{A}}$ constructs some gradient pairings $\alpha\nearrow\beta$ earlier than they would be constructed by $\mathcal{A}$, thus avoiding the need to perform subsequent testing instructions related to $\alpha$ or $\beta$.

\begin{example}\label{pinched}{\em
Consider the triangulation of the punctured projective plane shown in Figure~\ref{ponchado}. In the algorithm~$\mathcal{A}$, the pairing $[2,3]\nearrow[2,3,4]$, which is constructed during the process $\mathcal{P}^{1,2,4}$, comes before the pairing $[1,5]\nearrow[1,4,5]$, which is constructed during the process $\mathcal{P}^{1,1,4}$. Instead, these two pairings arise in the opposite order in the algorithm $\overline{\mathcal{A}}$, and they both are constructed during the process~$\mathcal{Q}^{1,4}$. As the reader can easily check, the (common) resulting gradient field has only two critical faces, namely $[6]$ and $[4,5]$, and is thus optimal (for the punctured projective plane has the homotopy type of the circle~$S^1$).
}\end{example}

\begin{figure}
\centering
\begin{tikzpicture}[x=1cm,y=1cm]
\node[left] at (0,0) {2};\node[left] at (1,1.5) {3};\node[above] at (2,3) {6};
\node[right] at (3,1.5) {2};\node[right] at (4,0) {3};\node[below] at (2,0) {6};
\node[right] at (2,1.75) {5};\node[below] at (1.46,.75) {4};\node[below] at (2.54,.75) {1};
\draw[very thin](0,0)--(4,0);\draw[very thin](0,0)--(2,3);\draw[very thin](4,0)--(2,3);
\draw[very thin](1,1.5)--(3,1.5);\draw[very thin](1,1.5)--(2,0);\draw[very thin](3,1.5)--(2,0);
\draw[very thin](2,1.5)--(2,3);
\draw[very thin](0,0)--(1.5,.75);
\draw[very thin](4,0)--(2.5,.75);
\draw[very thin](2,1.5)--(1.5,.75);
\draw[very thin](2,1.5)--(2.5,.75);
\draw[very thin](1.5,.75)--(2.5,.75);
\coordinate (Origin) at (0,0);
\filldraw [thick, fill=gray] ($(Origin)+(3.96,.04)$) -- ++(-1.44,.72) -- ++(.48,.72) -- cycle;
\end{tikzpicture}
\caption{Vertex order in the projective plane with facet $[1,2,3]$ is removed}
\label{ponchado}
\end{figure}
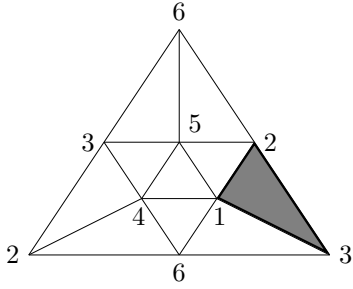
 
The goal of this subsection is to prove Theorem~\ref{mismoscampos} below, i.e., the fact that $W=\overline{W}$ at the end of both algorithms. The proof is best organized by setting $\overline{W}_{i,r,v}:=P_{i,r,v}\cap\overline{W}$ (cf.~Definition~\ref{parespotenciales}), as well as
$$W_{k,v}=\bigsqcup_r W_{k,r,v}\ \ \mbox{ and } \ \ \overline{W}_{k,v}=\bigsqcup_r\overline{W}_{k,r,v}.$$

\begin{theorem}\label{mismoscampos}
The pairings constructed by $\mathcal{A}$ and $\overline{\mathcal{A}}$ agree: $W_{k,r,v}=\overline{W}_{k,r,v}$ for all relevant indices $k$, $r$ and $v$. In particular, $\overline{W}$ is acyclic.
\end{theorem}

The proof of Theorem~\ref{mismoscampos} makes use of the following elementary observations holding for vertices $v$ and $w$ with $v\preceq w$:
\begin{align}\label{elementales}
\begin{gathered}
(\alpha,\beta)\in P_{k,r,v}\mbox{ \ and \ } (\alpha,\gamma)\in P_{k,s,w} \implies r\leq s, \mbox{ with equality provided $v=w$.}\\
(\alpha,\beta)\in P_{k,r,v}\mbox{ \ and \ } (\gamma,\beta)\in P_{k,s,w} \implies r\leq s, \mbox{ with equality provided $v=w$.}
\end{gathered}
\end{align}
Additionally, it will be convenient to keep in mind the following closer view of the medular part of algorithms $\mathcal{A}$ and~$\overline{\mathcal{A}}$. Namely, in the case of $\mathcal{A}$, an efficient way to execute a process $\mathcal{P}^{k,r,v}$ is by assembling the set $N_{k,r,v}$ of $(k+1)$-dimensional faces $\gamma$ having $v$ in position $r$ and so that both $\gamma$ and $\partial_v(\gamma)$ ---the latter being the face of~$\gamma$ obtained by removing $v$--- are available, i.e., neither $\gamma$ nor $\partial_v(\gamma)$ have been paired previous to the start of $\mathcal{P}^{k,r,v}$. With such a preparation, $\mathcal{P}^{k,r,v}$ simply adds\footnote{The adding of pairs is done following the $\preceq$-lexicographic order (cf.~Lemma~\ref{lexi}), thought this much is immaterial at this point.} to $W$ all pairs $(\partial_v(\gamma),\gamma)$ with $\gamma\in N_{k,r,v}$ (construction of new pairings), and removes all faces $\gamma$ and $\partial_v(\gamma)$, for $\gamma\in N_{k,r,v}$, from the corresponding lists of unpaired faces (update of available faces). Likewise, an efficient way to execute process $\overline{\mathcal{P}}^{k,v}$ in $\overline{\mathcal{A}}$ is by assembling the set $\overline{N}_{k,v}$ of $(k+1)$-dimensional faces $\gamma$ containing $v$ as a vertex (in any position) and such that both $\gamma$ and $\partial_v(\gamma)$ are available at the start of $\overline{\mathcal{P}}^{k,v}$. With such a preparation, Lemma~\ref{lexi} below shows that $\overline{\mathcal{P}}^{k,v}$ simply adds to $\overline{W}$ (in lexicographic order) all pairings $(\partial_v(\gamma),\gamma)$ with $\gamma\in\overline{N}_{k,v}$ (construction of new pairings), and removes all faces $\gamma$ and $\partial_v(\gamma)$, for $\gamma\in \overline{N}_{k,v}$, from the corresponding lists of unpaired faces (update of available faces). In particular, $N_{k,r,v}$ (respectively $\overline{N}_{k,v}$) is the set of collapsible faces for the block of pairing constructed by $\mathcal{P}^{k,r,v}$ (respectively $\overline{\mathcal{P}}^{k,v}$), while the faces $\partial_v(\beta)$ for $\beta\in N_{k,r,v}$ (respectively $\beta\in\overline{N}_{k,v}$) are the corresponding redundant faces.

\begin{lemma}\label{lexi}
Let $\alpha$ and $\beta$ be $(k+1)$-dimensional faces each containing $v$ as a vertex (in any position). In terms of the $\preceq$-lexicographic order, the condition $\alpha\prec\beta$ holds if and only if $\partial_v(\alpha)\prec\partial_v(\beta)$.
\end{lemma}
\begin{proof}
We provide proof details for completeness. The lexicographic order is linear (we have assumed so at the vertex level), so it suffices to show that $\partial_v(\alpha)\prec\partial_v(\beta)$ provided $\alpha\prec\beta$. Say $v$ appears in positions $i$ and $j$ in $\alpha$ and $\beta$, respectively. The result is obvious if $i=j$ (this is why we did not need the lemma in our closer look at $\mathcal{A}$), or if the lexicographic decision for the inequality $\alpha\prec\beta$ is taken at a position smaller than $m:=\min\{i,j\}$. Thus, we can assume $i\neq j$ with $\alpha$ and $\beta$ being identical up to position $m-1$. The inequality $\alpha\prec\beta$ then forces $i>j=m$. Thus $\partial_v(\alpha)$ and $\partial_v(\beta)$ are identical up to position $j-1$, while in position $j$:
\begin{itemize}
\item $\partial_v(\alpha)$ has the vertex $\alpha_j$, which is smaller than $v=\alpha_i$, and
\item $\partial_v(\beta)$ has the vertex $\beta_{j+1}$, which is larger than $v=\beta_j$.
\end{itemize}
Consequently $\partial_v(\alpha)\prec\partial_v(\beta)$.
\end{proof}

\begin{proof}[Proof of Theorem~\ref{mismoscampos}]
Recall $d$ stands for the dimension of the simplicial complex under consideration. Fix $i\in\{0,1,\ldots,d-1\}$ and assume
\begin{equation}\label{externalinduction}
\mbox{the equality $W_{k,r,v}=\overline{W}_{k,r,v}$ is valid whenever $k>i$}
\end{equation}
for all relevant values of $r$ and $v$. The inductive goal is to prove
\begin{equation}\label{items}
W_{i,r,v}=\overline{W}_{i,r,v}, \mbox{ for all $v\in V$ and all $r\in\{0,1,\ldots,i+1\}$}.
\end{equation}
(The induction is vacuously grounded by the fact that $W_{d,r,v}=\varnothing=\overline{W}_{d,r,v}$ at the start of both algorithms.) We start by arguing the case $r=i+1$ in~(\ref{items}) which, in turn, will be done by induction on the reverse ordering of vertices (i.e., starting from the largest vertex $v_{\max}$) and through a comparison of the corresponding actions of $\mathcal{A}$ and $\overline{\mathcal{A}}$ during \emph{simultaneous} execution of these algorithms. In detail:

\smallskip{\bf Case I: $r=i+1$ and $v=v_{\max}$ in~(\ref{items})}. Pairings in $W_{i,i+1,v_{\max}}$ are constructed during the execution of process $\mathcal{P}^{i,i+1,v_{\max}}$, while those in $\overline{W}_{i,i+1,v_{\max}}$ are constructed during the execution of process $\overline{\mathcal{P}}^{i,v_{\max}}$. In principle, the latter process would also construct pairings outside $\overline{W}_{i,i+1,v_{\max}}$. However, such a possibility is prevented from the fact that $v_{\max}$ can only appear in the last position of any face. Taking into account the inductive assumption~(\ref{externalinduction}), this means that processes $\mathcal{P}^{i,i+1,v_{\max}}$ in $\mathcal{A}$ and $\overline{\mathcal{P}}^{i,v_{\max}}$ in $\overline{\mathcal{A}}$ construct the same new pairings and, consequently, perform the same updating of sets of available faces (this justifies the abuse of notation $\mathcal{P}^{i,i+1,v_{\max}}=\overline{\mathcal{P}}^{i,v_{\max}}$). Furthermore, after these processes conclude, no further pairings can be constructed by insertion of $v_{\max}$ (either in $\mathcal{A}$ or in $\overline{\mathcal{A}}$). Thus in fact
\begin{equation}\label{bloque1}
W_{i,v_{\max}}=W_{i,i+1,v_{\max}}=\overline{W}_{i,i+1,v_{\max}}=\overline{W}_{i,v_{\max}}
\end{equation}
which, in particular, grounds the inductive (on the vertices) argument for the case $r=i+1$ in~(\ref{items}). As explained in the paragraph preceding Lemma~\ref{lexi}, the redundant entries in~(\ref{bloque1}) are the $i$-dimensional faces $\alpha$ such that $\alpha\cup\{v_{\max}\}$ is an $(i+1)$-dimensional face (so $v_{\max}\not\in\alpha$) available at the start of process $\mathcal{P}^{i,i+1,v_{\max}}=\overline{\mathcal{P}}^{i,v_{\max}}$ (all $i$-dimensional faces are available at this point), whereas the collapsible entries in~(\ref{bloque1}) are the $(i+1)$-dimensional faces available at the start of $\mathcal{P}^{i,i+1,v_{\max}}=\overline{\mathcal{P}}^{i,v_{\max}}$ that contain $v_{\max}$ as a vertex.

The above situation changes slightly in later stages of the algorithms and, in order to better appreciate subtleties, it is highly illustrative to spend a little time analyzing in detail a few of the pairings constructed right after~(\ref{bloque1}).

\smallskip{\bf Case II: $r=i+1$ and $v=v_{\max-1}$ in~(\ref{items})}. Let $v_1,v_2,\ldots,v_{\max-1},v_{\max}$ be the elements of the vertex set $V$ listed increasingly according to $\preceq$. Pairings in $W_{i,i+1,v_{\max-1}}$ (respectively $\overline{W}_{i,i+1,v_{\max-1}}$) are constructed during the execution of the process $\mathcal{P}^{i,i+1,v_{\max-1}}$ (respectively $\overline{\mathcal{P}}^{i,v_{\max-1}}$). In both processes, the construction is done by considering insertion of $v_{\max-1}$ among available faces (these are common to both algorithms up to this point), either in position $i+1$ in the case of $\mathcal{A}$, or in any position in the case of $\overline{\mathcal{A}}$. As in Case I,
\begin{equation}\label{posibilidad}
\mbox{$\overline{\mathcal{P}}^{i,v_{\max-1}}$ might construct pairings outside $\overline{W}_{i,i+1,v_{\max-1}}$,}
\end{equation}
and
\begin{equation}\label{nexttolast}
\mbox{any such a pairing would have to lie in $\overline{W}_{i,i,v_{\max-1}}$,}
\end{equation}
as $v_{\max-1}$ cannot appear in a position smaller than~$i$ in an $(i+1)$-dimensional face. In terms of the notation introduced in the paragraph previous to Lemma~\ref{lexi}, the possibility in~(\ref{posibilidad}) translates into a strict inclusion $N_{i,i+1,v_{\max-1}}\subset\overline{N}_{i,v_{\max-1}}$. However an element in $\overline{N}_{i,v_{\max-1}}\setminus N_{i,i+1,v_{\max-1}}$ is forced to be an $(i+1)$-dimensional face which, in addition to being available at the start of $\mathcal{P}^{i,i+1,v_{\max-1}}$ and $\overline{\mathcal{P}}^{i,v_{\max-1}}$, has $v_{\max}$ appearing in the last position (for, as indicated in~(\ref{nexttolast}), $v_{\max-1}$ appears in the next-to-last position). Such a situation conflicts with the description of collapsible faces noted at the end of Case I, ruling out the possibility in~(\ref{posibilidad}). Thus, as above, $\mathcal{P}^{i,i+1,v_{\max-1}}=\overline{\mathcal{P}}^{i,v_{\max-1}}$ and 
\begin{equation}\label{bloque2}
W_{i,v_{\max-1}}=W_{i,i+1,v_{\max-1}}=\overline{W}_{i,i+1,v_{\max-1}}=\overline{W}_{i,v_{\max-1}}.
\end{equation}

While Cases I and II are essentially identical, the construction of subsequent pairings has a twist whose solution is better appreciated by taking a quick glance at the next block of pairings, i.e., those constructed by $\mathcal{P}^{i,i+1,v_{\max-2}}$ in the case of $\mathcal{A}$ and by $\overline{\mathcal{P}}^{i,v_{\max-2}}$ in the case of $\overline{\mathcal{A}}$. Namely, this time the inclusion
\begin{equation}\label{twist}
N_{i,i+1,v_{\max-2}}\subseteq\overline{N}_{i,v_{\max-2}}
\end{equation}
may actually fail to be an equality, as illustrated in Example~\ref{pinched}. As a result, the particularly strong form of assertions~(\ref{bloque1}) and~(\ref{bloque2}) no longer holds true for subsequent blocks of pairings. In any case, what we do recover from~(\ref{twist}) ---and the discussion previous to Lemma~\ref{lexi}--- is the fact that $W_{i,i+1,v_{\max-2}}=\overline{W}_{i,i+1,v_{\max-2}}$. We next extend inductively the latter conclusion to other vertices and, then, explain how early pairings constructed in $\overline{\mathcal{A}}$ are eventually recovered in $\mathcal{A}$. 

\smallskip{\bf Case III: Inductive step settling~(\ref{items}) for $r=i+1$}. Fix a vertex $v\in V$ and assume
\begin{equation}\label{pasind}
\overline{W}_{i,i+1,w}=W_{i,i+1,w}
\end{equation}
whenever $v\prec w$, allowing the possibility that process $\overline{\mathcal{P}}^{i,w}$ in $\overline{\mathcal{A}}$ constructs more pairings than those constructed by the corresponding process $\mathcal{P}^{i,i+1,w}$ in $\mathcal{A}$. In such a setting, faces available at the start of $\overline{\mathcal{P}}^{i,v}$ are necessarily available at the start of $\mathcal{P}^{i,i+1,v}$, so
\begin{equation}\label{wsiguales}
\overline{W}_{i,i+1,v}\subseteq W_{i,i+1,v}.
\end{equation}
Assume for a contradiction that the latter inclusion is strict, and pick a pairing
\begin{equation}\label{pick}
\mbox{$(\alpha,\beta)$ in $W_{i,i+1,v}$ not in $\overline{W}_{i,i+1,v}$.}
\end{equation}
This means that $\alpha$ or $\beta$ (or both) are not available at the start of $\overline{\mathcal{P}}^{i,v}$ and, in view of~(\ref{externalinduction}), this can only happen provided either
\begin{enumerate}[(i)]
\item $(\alpha,\beta')\in \overline{W}_{i,r,w}\subseteq P_{i,r,w}$ for some face $\beta'$, some vertex $w$ and some position $r$, or 
\item $(\alpha',\beta)\in \overline{W}_{i,r,w}\subseteq P_{i,r,w}$ for some face $\alpha'$, some vertex $w$ and some position $r$,
\end{enumerate}
where, in either case, $v\prec w$ and $r\leq i+1$. But $(\alpha,\beta)\in W_{i,i+1,v}\subseteq P_{i,i+1,v}$, so~(\ref{elementales}) yields in fact $i+1=r$. Thus $\alpha$ or $\beta$ is part of a pairing in $\overline{W}_{i,r,w}=\overline{W}_{i,i+1,w}=W_{i,i+1,w}$, where the latter equality comes from~(\ref{pasind}) but contradicts~(\ref{pick}). Thus~(\ref{wsiguales}) is an equality. Note that the above argument does not rule out the possibility that $\overline{\mathcal{P}}^{i,v}$ constructs more pairings (by inserting $v$ at a position smaller than $i+1$) than those constructed by $\mathcal{P}^{i,i+1,v}$.

The conclusion of the proof of Theorem~\ref{mismoscampos}, i.e., the proof of~(\ref{items}) for $r\leq i$, proceeds by (inverse) induction on $r$, with the above discussion for $r=i+1$ grounding the induction. The new inductive argument requires an entirely different viewpoint coming from the following fact: In $\mathcal{A}$, after $\mathcal{P}^{i,i+1,v_1}$ is over, process $\mathcal{P}^i$ continues with many more subprocesses, the first of which is $\mathcal{P}^{i,i,v_{\max}}$. Yet, in $\overline{\mathcal{A}}$, process $\overline{\mathcal{P}}^i$ finishes as soon $\overline{\mathcal{P}}^{i,v_1}$ is over, i.e., when the final inductive stage in Case III concludes. Therefore, our proof strategy from this point on requires pausing $\overline{\mathcal{A}}$ in order to analyze the rest of the actions in $\mathcal{P}^i$. In particular, we explain next how $\mathcal{P}^i$ catches up with all the ``early'' pairings $\bigcup_v\left(\overline{W}_{i,v}\setminus W_{i,i+1,v}\right)$ constructed by $\overline{\mathcal{P}}^i$.

\smallskip{\bf Case IV: (Double) inductive step settling~(\ref{items}) for any $r$}. Fix $r\in\{0,1,\ldots,i\}$ and assume inductively that, as $\mathcal{P}^i$ progresses, $\mathcal{P}^{i,\rho}$ yields $\overline{W}_{i,\rho,w}=W_{i,\rho,w}$ for any vertex $w$ and any position of insertion $\rho>r$. (The induction is grounded by Case III above.) The goal is to prove
\begin{equation}\label{dobleindext}
\mbox{$\overline{W}_{i,r,w}=W_{i,r,w}$ for all vertices $w$.}
\end{equation}
Since $r\leq i$, we get $\overline{W}_{i,r,v_{\max}}=\varnothing=W_{i,r,v_{\max}}$. We can therefore assume in a second inductive level that, for some vertex $v$ with $v\prec v_{\max}$,~(\ref{dobleindext}) holds true for all vertices $w$ with $v\prec w$. The updated goal is to prove $\overline{W}_{i,r,v}=W_{i,r,v}$.

\smallskip\emph{Inclusion $\overline{W}_{i,r,v}\subseteq W_{i,r,v}\hspace{.3mm}$}: Suppose for a contradiction that
\begin{equation}\label{contr1}
(\alpha,\beta)\in\overline{W}_{i,r,v}
\end{equation}
is an ``early'' pairing (constructed during the execution of $\overline{\mathcal{P}}^{i,v}$) that cannot be constructed during the execution of $\mathcal{P}^{i,r,v}$. Then $\alpha$ or $\beta$ (or both) must be involved as a pairing of some $W_{i,s,w}$ with $s\geq r$ and, in addition, with $v\prec w$ if in fact $s=r$. The double inductive equality $W_{i,s,w}=\overline{W}_{i,s,w}$, the dynamics of $\overline{\mathcal{A}}$ and~(\ref{contr1}) then force $v=w$ and, consequently, $s>r$. But the latter inequality contradicts~(\ref{elementales}) since $\overline{W}_{i,r,v}\subseteq P_{i,r,v}$ and $\overline{W}_{i,s,w}\subseteq P_{i,s,w}$.

\smallskip\emph{Inclusion $W_{i,r,v}\subseteq\overline{W}_{i,r,v}\hspace{.3mm}$}: Suppose for a contradiction that
\begin{equation}\label{contr2}
(\alpha,\beta)\in W_{i,r,v}
\end{equation}
is not one of the ``early'' pairings constructed during the execution of $\overline{\mathcal{P}}^{i,v}$. Then $\alpha$ or $\beta$ (or both) must be involved in a pairing of some $\overline{W}_{i,s,w}$ with $v\prec w$. As in the previous paragraph, (\ref{elementales}) then yields $r\leq s$. In turn, the double inductive hypothesis gives $\overline{W}_{i,s,w}=W_{i,s,w}$, which thus contains a pairing involving $\alpha$ or $\beta$, in contradiction to~(\ref{contr2}).
\end{proof}

Despite $\overline{\mathcal{A}}$ is faster than $\mathcal{A}$, in general there is a high computational cost involved in running either of the two algorithms. So, rather than estimating computational complexity issues, in the next section we describe concrete ``local'' conditions that allow us to identify gradient pairings. In a number of instances\footnote{This holds, for instance, in the case of the projective plane and the torus in Examples~\ref{projectiveandtorus}, as well as in the application to spaces of ordered pairs of points on complete graphs in Subsecci\'on~\ref{aplimunk}.}, the conditions determine in full the gradient field.

\subsection{Collapsibility conditions}\label{colacond}
In this section we identify a set of conditions implying collapsibility of a given face. As we shall see later in the paper, this determines in full the gradient field constructed by $\mathcal{A}$ in the case of Munkres' combinatorial model for a configuration space of ordered pairs of points (Section~\ref{munkres-model}) on a complete graph. The main result (Theorem~\ref{extended}) is presented through a series of preliminary complexity-increasing results in order to isolate the role of each of the condition ingredients.

\begin{definition}\label{closing-vertex}
A vertex $\alpha_i$ of a face $\alpha=[\alpha_0,\ldots,\alpha_k]\in\mathcal{F}^k$ is said to be maximal in $\alpha$ if $\partial_{\alpha_i}(\alpha)\cup\{v\}\notin\mathcal{F}^{k}$ for all vertices $v$ with $\alpha_i\prec v$. When $\alpha_i$ is non-maximal in $\alpha$, we write $\alpha(i):=\partial_{\alpha_i}(\alpha)\cup\{\alpha^i\}$, where
$$
\alpha^i:=\max\{v\in V \colon \alpha_i\prec v \mbox{ and }\hspace{.3mm}\partial_{\alpha_i}(\alpha)\cup\{v\}\in\mathcal{F}^k\}.
$$
Note that $\alpha^i$ is maximal in $\alpha(i)$, and that $\alpha^i$ is not a vertex of $\alpha$. Iterating the construction, for a given $\alpha=[\alpha_0,\ldots,\alpha_k]\in\mathcal{F}^k$ and a sequence of integers $0\leq i_1<i_2<\cdots<i_p\leq k$, we say that the ordered vertices $\alpha_{i_1},\alpha_{i_2},\ldots,\alpha_{i_p}$ are non-maximal in $\alpha$ provided:
\begin{itemize}
\item $\alpha_{i_1}$ is non-maximal in $\alpha$, so we can form the face $\alpha(i_1)$;
\item $\alpha_{i_2}$ is non-maximal in $\alpha(i_1)$, so we can form the face $\alpha(i_1,i_2):=\alpha(i_1)(i_2)$;
\item \ldots
\item $\alpha_{i_p}$ is non-maximal in $\alpha(i_1,\ldots,i_{p-1})$, so we can form the face $\alpha(i_1,\ldots,i_p):=\alpha(i_1,\ldots,i_{p-1})(i_p)$.
\end{itemize}
When $p=0$ (so there is no constructing process), $\alpha(i_1,i_2,\ldots,i_p)$ is interpreted as $\alpha$.
\end{definition}

\begin{lemma}\label{noredundant}
No vertex of a redundant $k$-face $\alpha\in\mathcal{F}^k$ is maximal in $\alpha$.
\end{lemma}
\begin{proof}
Assume a pairing $\alpha=[\alpha_0,\ldots,\alpha_{k}]\nearrow\beta=[\alpha_0,\ldots,\alpha_{r-1},\beta_0,\alpha_r,\ldots,\alpha_{k}]$ and consider a vertex $\alpha_i$ of $\alpha$. If $i<r$, the $k$-face $\partial_{\alpha_i}(\beta)=[\alpha_0,\ldots,\widehat{\alpha_i},\ldots,\alpha_{r-1},\beta_0,\alpha_r,\ldots,\alpha_{k}]$ shows that $\alpha_i$ is non-maximal in $\alpha$. If $i\geq r$, Lemma~\ref{apareaconlamenor} gives a pairing
$$
\gamma:=[\alpha_0,\ldots,\alpha_{r-1},\beta_0,\alpha_r,\ldots,\widehat{\alpha_i},\ldots,\alpha_{k}]\nearrow [\alpha_0,\ldots,\alpha_{r-1},\beta_0,\alpha_r,\ldots,\widehat{\alpha_i},\ldots,\delta_0,\ldots]=:\delta
$$
by insertion of a vertex $\delta_0$ with $\alpha_i\prec\delta_0$, so that the $k$-face $\partial_{\beta_0}(\delta)$ shows that $\alpha_i$ is non-maximal in~$\alpha$.
\end{proof}

While maximal vertices in a face $\alpha$ can be thought of as giving obstructions for redundancy of~$\alpha$, maximality of the largest vertex in $\alpha$ is actually equivalent to collapsibility of $\alpha$ in a specific way:

\begin{corollary}\label{collapsible}
The following conditions are equivalent for a $k$-face $\alpha=[\alpha_0,\ldots,\alpha_k]\in\mathcal{F}^k:$
\begin{enumerate}[(1)]
\item\label{maxbis} $\alpha_k$ is maximal in $\alpha$.
\item\label{maxmax} $\partial_{\alpha_k}(\alpha)\nearrow\alpha$.
\end{enumerate}
\end{corollary}
\begin{proof}
Assuming~\emph{(\ref{maxbis})}, both $\alpha$ and $\partial_{\alpha_k}(\alpha)$ are available at the start of process $\mathcal{P}^{k-1,k,\alpha_k}$; the former face in view of Lemma~\ref{noredundant}, and the latter face by the maximality hypothesis. The $W$-pairing in~\emph{(\ref{maxmax})} is therefore constructed by the process $\mathcal{P}^{k-1,k,\alpha_k}$. On the other hand, if~\emph{(\ref{maxbis})} fails, there is a vertex $v$ of $K$ which is maximal with respect to the conditions $\alpha_k\prec v$ and $\partial_{\alpha_k}(\alpha)\cup\{v\}\in\mathcal{F}^k$. As $v$ is maximal in $\partial_{\alpha_k}(\alpha)\cup\{v\}=[\alpha_0,\ldots,\alpha_{k-1},v]$, the argument in the previous paragraph gives $[\alpha_0,\ldots,\alpha_{k-1}]\nearrow [\alpha_0,\ldots,\alpha_{k-1},v]$, thus ruling out the $W$-pairing in~\emph{(\ref{maxmax})}.
\end{proof}

Under additional restrictions (spelled out in~(\ref{additionalrestr}) below), maximality of other vertices also forces collapsibility in a specific way. We start with the case of the next-to-last vertex, where the additional restrictions are simple, yet the if-and-only-if situation in Corollary~\ref{collapsible} is lost (see Remark~\ref{noanalogue} below).

\begin{proposition}\label{startofgeneral-collapsible}
Let $\alpha=[\alpha_0,\ldots,\alpha_k]\in\mathcal{F}^k$. If $\alpha_{k-1}$ is maximal in $\alpha$ but $\alpha_k$ is not, then $\partial_{\alpha_{k-1}}(\alpha)\nearrow\alpha$.
\end{proposition}
\begin{proof}
By Lemma~\ref{noredundant}, $\alpha$ is available at the start of process $\mathcal{P}^{k-1}$ and, in fact, at the start of process $\mathcal{P}^{k-1,k-1,\alpha_{k-1}}$, in view of Corollary~\ref{collapsible} and the hypothesis on $\alpha_k$. The asserted pairing follows since $\partial_{\alpha_{k-1}}(\alpha)=[\alpha_0,\ldots,\widehat{\alpha_{k-1}},\alpha_k]$ is also available at the start of process $\mathcal{P}^{k-1,k-1,\alpha_{k-1}}$. Indeed, a potential pairing $\partial_{\alpha_{k-1}}(\alpha)\nearrow\partial_{\alpha_{k-1}}(\alpha)\cup\{v\}$ constructed at a stage before $\mathcal{P}^{k-1,k-1,\alpha_{k-1}}$ would have $\alpha_{k-1}\prec v$, contradicting the maximality of $\alpha_{k-1}$ in $\alpha$.
\end{proof}

\begin{remark}\label{noanalogue}{\em
Consider the gradient field on the projective plane in Examples~\ref{projectiveandtorus}. Neither $5$ nor $2$ are maximal in $[1,2,5]$ (due to the faces $[1,2,6]$ and $[1,3,5]$), yet the pairing $[1,5]\nearrow[1,2,5]$ holds.
}\end{remark}

More generally,
\begin{proposition}\label{general-ahora-si}
For a face $\alpha=[\alpha_0,\ldots,\alpha_k]\in\mathcal{F}^k$ and an integer $r\in\{0,1,\ldots,k\}$ with $\alpha_r$ maximal in $\alpha$, the pairing $\partial_{\alpha_r}(\alpha)\nearrow\alpha$ holds provided
\begin{equation}\label{additionalrestr}
\mbox{for any sequence $r+1\leq t_1<\cdots<t_p\leq k$, the ordered vertices $\alpha_{t_1},\ldots,\alpha_{t_p}$ are non-maximal in $\alpha$.}
\end{equation}
\end{proposition}
\begin{proof}
We argue by decreasing induction on $r=k,k-1,\ldots,0$. The grounding cases $r=k$ and $r=k-1$ are covered by Corollary~\ref{collapsible} and Proposition~\ref{startofgeneral-collapsible}, respectively. For the inductive step, the maximality of $\alpha_r$ in~$\alpha$ assures both that $\alpha$ is available at the start of $\mathcal{P}^{k-1}$ (Lemma~\ref{noredundant}), and that $\partial_{\alpha_r}(\alpha)$ is available at the start of $\mathcal{P}^{k-1,r,\alpha_r}$. It thus suffices to note that~(\ref{additionalrestr}) implies that $\alpha$ is also available at the start of $\mathcal{P}^{k-1,r,\alpha_r}$. But a potential pairing $[\alpha_0,\ldots,\widehat{\alpha_{t_1}},\ldots,\alpha_k]\nearrow[\alpha_0,\ldots,\alpha_k]$ previous in $\mathcal{A}$ to the intended pairing $\partial_{\alpha_r}(\alpha)\nearrow\alpha$, i.e., with $t_1\in\{r+1,\ldots,k\}$ is inductively ruled out by the (yet previous in $\mathcal{A}$) pairing $$[\alpha_0,\ldots,\widehat{\alpha_{t_1}},\ldots,\alpha_k]\nearrow\alpha(t_1)=[\alpha_0,\ldots,\widehat{\alpha_{t_1}},\{\alpha^{t_1},\alpha_{t_1+1},\ldots,\alpha_k\}],$$ where the use of curly braces is meant to indicate that $\alpha^{t_1}$ may occupy any position among the ordered vertices $\alpha_{t_1+1},\ldots,\alpha_k$.
\end{proof}

Not all conditions in~(\ref{additionalrestr}) would be needed in concrete instances of Proposition~\ref{general-ahora-si}. For instance, this will be (recursively) the case if, in the previous proof, some $\alpha^{t_1}$ turns out to be larger than some of the vertices $\alpha_{t_1+1},\ldots,\alpha_k$. 

\begin{example}\label{antepenultima}{\em
The pairing $\partial_{\alpha_{k-2}}(\alpha)\nearrow\alpha=[\alpha_0,\ldots,\alpha_k]$ holds provided \emph{(i)} $\alpha_{k-2}$ is maximal in $\alpha$, \emph{(ii)} $\alpha_{k-1}$ is non-maximal in $\alpha$, and \emph{(iii)} $\alpha_k$ is non-maximal in $\alpha$ as well as in $\alpha(k-1)$. Note that~\emph{(ii)} is used in order to state~\emph{(iii)}.
}\end{example}

Theorem~\ref{extended} below, a far-reaching extension of Proposition~\ref{general-ahora-si}, provides sufficient conditions that allow us to identify ``exceptional'' pairings such as the one noted in Remark~\ref{noanalogue}.

\begin{definition}\label{collapsing}
A vertex $\alpha_r$ of a face $\alpha=[\alpha_0,\ldots,\alpha_k]\in\mathcal{F}^k$ is said to be collapsing in $\alpha$ provided \emph{(i)} the face $\alpha$ is not redundant, \emph{(ii)} condition~\emph{(\ref{additionalrestr})} holds and \emph{(iii)} for every $v$ with $\alpha_r\prec v$ and $\partial_{\alpha_r}(\alpha)\cup\{v\}\in\mathcal{F}^k$, there is a vertex $\alpha_j$ of $\alpha$ with $v\prec\alpha_j$ such that $\alpha_j$ is collapsing in $\partial_{\alpha_r}(\alpha)\cup\{v\}$.
\end{definition}

The first and third conditions in Definition~\ref{collapsing} hold when $\alpha_r$ is maximal in $\alpha$. Note the recursive nature of Definition~\ref{collapsing}.

\begin{theorem}\label{extended}
If $\alpha_r$ is collapsing in $\alpha$, then $\partial_{\alpha_r}(\alpha)\nearrow\alpha$.
\end{theorem}
\begin{proof}
The proof is parallel to that of Proposition~\ref{general-ahora-si}. This time the induction is grounded by Corollary~\ref{collapsible} and the observation that, when $r=k$, condition~\emph{(iii)} in Definition~\ref{collapsing} implies in fact that $\alpha_k$ is maximal in~$\alpha$. The rest of the argument in the proof of Proposition~\ref{general-ahora-si} applies with two minor adjustments. First, Lemma~\ref{noredundant} is not needed ---neither can it be applied--- in view of condition~\emph{(i)}). Second, the fact that $\partial_{\alpha_r}(\alpha)$ is available at the start of $\mathcal{P}^{k-1,r,\alpha_r}$ comes directly from~\emph{(iii)} and induction.
\end{proof}

\section{Application to configuration spaces}\label{aplicacion}
We use the gradient field in the previous section in order to describe the cohomology ring of the configuration space of ordered pairs of points on a complete graph.

\subsection{Gradient field on Munkres' homotopy simplicial model}\label{aplimunk}
Let $K_m$ be the 1-dimensional skeleton of the full $(m-1)$-dimensional simplex on vertices $V_m=\{1,2,\ldots,m\}$. Thus $|K_m|$ is the complete graph on the $m$ vertices. The homotopy type of $\C(|K_m|,2)$ is well understood for $m\leq3$, so we assume $m\geq4$ from now on. We think of $K_m$ as an ordered simplicial complex with the natural order on $V_m$, and study $\C(|K_m|,2)$ through its simplicial homotopy model $C_m$ in~(\ref{homotopy-equivalence}). The condition $m\geq4$ implies that $C_m$ is a pure 2-dimensional complex, i.e., all of its maximal faces have dimension 2. Furthermore, 2-dimensional faces of $C_m$ have one of the forms
\begin{equation}\label{matrix-type-notation}
\left[\begin{matrix}
a & a & d \\
b & c & c
\end{matrix}\right]
\qquad\mbox{or}\qquad
\left[\begin{matrix}
a' & c' & c' \\
b' & b' & d'
\end{matrix}\right]
\end{equation}
where
\begin{equation}\label{condiciones1}
\mbox{$d>a\notin\{b,c\}$, \;\;$b<c\neq d$, \;\;$d'>b'\notin\{a',c'\}$ \;\;and \;\;$a'<c'\neq d'$.}
\end{equation}
Note that the matrix-type notation in~(\ref{matrix-type-notation}) is compatible with the notation $\alpha=[\alpha_0,\ldots,\alpha_k]$ in previous sections; each $\alpha_i$ now stands for a column-type vertex $\genfrac{}{}{0pt}{1}{a}{b}$ (with $a\neq b$). In what follows, the conditions in~(\ref{condiciones1}) on the integers $a,b,c,d,a',b',c',d'\in V_m$ will generally be implicit and omitted when writing a 2-simplex or one of its faces. For instance, the forced relations $a\neq b<d\neq a$ are omitted in item~(\ref{wn12}) of:

\begin{proposition}\label{Wn}
Let $W_m$ be the gradient field on $C_m$ constructed by the algorithm in Section~\ref{algorithm-field} with respect to the lexicographic order on the vertices $\genfrac{}{}{0pt}{1}{a}{b}=(a,b)\in V_m\times V_m\setminus\Delta_{V_m}$ of $C_m$. The full list of $W_m$-pairings is:
\begin{enumerate}[(a)]
\item\label{wn12} $\left[\genfrac{}{}{0pt}{1}{a}{b}\genfrac{}{}{0pt}{1}{a}{d}\right]\nearrow
\left[\genfrac{}{}{0pt}{1}{a}{b}\genfrac{}{}{0pt}{1}{a}{d}\genfrac{}{}{0pt}{1}{m}{d}\right]$, for $a<m>d$.

\item\label{wn3} $\left[
\genfrac{}{}{0pt}{1}{a}{b}
\genfrac{}{}{0pt}{1}{a}{m}
\right]\nearrow\left[
\genfrac{}{}{0pt}{1}{a}{b}
\genfrac{}{}{0pt}{1}{a}{m}
\genfrac{}{}{0pt}{1}{m-1}{m}
\right]$, for $a<m-1$.

\item\label{wn45} 
$\left[
\genfrac{}{}{0pt}{1}{a}{b}
\genfrac{}{}{0pt}{1}{c}{b}
\right]\nearrow\left[
\genfrac{}{}{0pt}{1}{a}{b}
\genfrac{}{}{0pt}{1}{c}{b}
\genfrac{}{}{0pt}{1}{c}{m}
\right]$, for $b<m>c$.

\item\label{wn6} 
$\left[
\genfrac{}{}{0pt}{1}{a}{b}
\genfrac{}{}{0pt}{1}{m}{b}
\right]\nearrow\left[
\genfrac{}{}{0pt}{1}{a}{b}
\genfrac{}{}{0pt}{1}{m}{b}
\genfrac{}{}{0pt}{1}{m}{m-1}
\right]$, for $b<m-1$.

\item\label{wn78} 
$\left[
\genfrac{}{}{0pt}{1}{a}{b}
\genfrac{}{}{0pt}{1}{c}{d}
\right]\nearrow\left[
\genfrac{}{}{0pt}{1}{a}{b}
\genfrac{}{}{0pt}{1}{c}{b}
\genfrac{}{}{0pt}{1}{c}{d}
\right]$, for $a<c$, $\hspace{.5mm}b<d$, $\hspace{.5mm}b\neq c\hspace{.5mm}$ and either $\hspace{.3mm}c<m>d\hspace{.7mm}$ or $\hspace{.7mm}c=m>d+1\hspace{.3mm}$.

\item\label{wn910} 
$\left[
\genfrac{}{}{0pt}{1}{a}{b}
\genfrac{}{}{0pt}{1}{c}{d}
\right]\nearrow\left[
\genfrac{}{}{0pt}{1}{a}{b}
\genfrac{}{}{0pt}{1}{a}{d}
\genfrac{}{}{0pt}{1}{c}{d}
\right]$, for $a<c$, $\hspace{.5mm}b<d$, $\hspace{.5mm}a\neq d$ \hspace{.5mm}and either $\hspace{.5mm}b=c<m>d\hspace{.7mm}$ or $\hspace{.7mm}c+1<m=d\hspace{.3mm}$.

\item\label{wn145} 
$\left[
\genfrac{}{}{0pt}{1}{a}{b}
\right]\nearrow\left[
\genfrac{}{}{0pt}{1}{a}{b}
\genfrac{}{}{0pt}{1}{m}{m-1}
\right]$, for either $\hspace{.5mm}b<m-1$ \hspace{.5mm}or\hspace{.7mm} $a<m-1=b$.

\item\label{wnpenultimo} 
$\left[
\genfrac{}{}{0pt}{1}{a}{m}
\right]\nearrow\left[
\genfrac{}{}{0pt}{1}{a}{m}
\genfrac{}{}{0pt}{1}{m-1}{m}
\right]$, for $a<m-1$.

\item\label{wnultimo} 
$\left[
\genfrac{}{}{0pt}{1}{m-1}{m}
\right]\nearrow\left[
\genfrac{}{}{0pt}{1}{m-1}{m-2}
\genfrac{}{}{0pt}{1}{m-1}{m}
\right]$.
\end{enumerate}
In particular, the critical faces are:
\begin{enumerate}[(a)]\addtocounter{enumi}{9}
\item\label{wnj} In dimension $0$, the vertex $\left[\genfrac{}{}{0pt}{1}{m}{m-1}\right]$.
\item\label{wnk} In dimension $1$, the simplices:
\begin{itemize}
\item[(k.1)] $\;\;\left[
\genfrac{}{}{0pt}{1}{a}{b}
\genfrac{}{}{0pt}{1}{m-1}{m}
\right]$, with either $a=m-1>b+1\hspace{.4mm}$ or $\hspace{.6mm}a<m-1\geq b$.
\item[(k.2)]
$\;\;\left[
\genfrac{}{}{0pt}{1}{m}{b}
\genfrac{}{}{0pt}{1}{m}{d}
\right]$, with $d<m-1$.
\item[(k.3)]
$\;\;\left[
\genfrac{}{}{0pt}{1}{a}{m}
\genfrac{}{}{0pt}{1}{c}{m}
\right]$, with $c<m-1$.
\end{itemize}
\item\label{wnl} In dimension $2$, the simplices 
$\left[
\genfrac{}{}{0pt}{1}{a}{b}
\genfrac{}{}{0pt}{1}{a}{d}
\genfrac{}{}{0pt}{1}{c}{d}
\right]$
with $\hspace{.3mm}b\neq c<m>d$.
\end{enumerate}
\end{proposition}

Note that the condition $a\neq d$ in item~\emph{(\ref{wn910})} above is forced to hold in the stronger form $a<d$. 

\begin{proof}
All pairings, except for the one in~\emph{(\ref{wn910})} when $b\neq c$ (so that $c+1<m=d\hspace{.3mm}$), are given by Corollary~\ref{collapsible} and Proposition~\ref{startofgeneral-collapsible}. The exceptional case requires the stronger Theorem~\ref{extended}. On the other hand, direct inspection shows that the faces listed as critical are precisely those not taking part in the list of $W_m$-pairings. The proof is then complete by observing that the criticality of any $d$-dimensional face $\alpha$ in \emph{(\ref{wnj})--(\ref{wnl})} is forced by the fact that all possible $(d-1)$-faces and all possible $(d+1)$-cofaces of $\alpha$ are involved in one of the pairings \emph{(\ref{wn12})--(\ref{wnultimo}).} For instance, a face $\left[
\genfrac{}{}{0pt}{1}{a}{b}
\genfrac{}{}{0pt}{1}{a}{d}
\genfrac{}{}{0pt}{1}{c}{d}
\right]$
in~\emph{(\ref{wnl})} is not collapsible since the three potential pairings 
$$
\left[\begin{matrix} a&a\\b&d\end{matrix}\right]\raisebox{-1.63mm}{\rotatebox{45}{$\dashrightarrow$}}\left[\begin{matrix} a&a&c\\b&d&d\end{matrix}\right],
\quad\;
\left[\begin{matrix} a&c\\b&d\end{matrix}\right]\raisebox{-1.63mm}{\rotatebox{45}{$\dashrightarrow$}}\left[\begin{matrix} a&a&c\\b&d&d\end{matrix}\right]
\quad\mbox{and}\quad
\left[\begin{matrix} a&c\\d&d\end{matrix}\right]\raisebox{-1.63mm}{\rotatebox{45}{$\dashrightarrow$}}\left[\begin{matrix} a&a&c\\b&d&d\end{matrix}\right]
$$
are ruled out by~\emph{(\ref{wn12})},~\emph{(\ref{wn78})} and~~\emph{(\ref{wn45})}, respectively. 
\end{proof}

The next-to-last sentence in the proof above reflects the maximality of $W_m$ ---see~(\ref{maxim}). On the other hand, a straightforward counting shows that the number $c_d$ of critical faces in dimension $d\in\{0,1,2\}$ is given by
\begin{equation}\label{rangos}
\mbox{$c_0=1$, \ $c_1=2(m-2)^2-1$ \ and \ $\hspace{.2mm}c_2=\frac{(m-1)(m-2)(m-3)(m-4)}{4}.$}
\end{equation}
In particular, the Euler characteristic of $\C(|K_m|,2)$ is given by $\frac{m(m^3-10m^2+27m-18)}{4}$, which yields an explicit expression for the conclusion of~\cite[Corollary~1.2]{MR2491587} in the case of complete graphs. Note in particular that the gradient field $W_4$ is optimal:
\begin{corollary}\label{caso4}
There is a homotopy equivalence $\C(|K_4|,2)\simeq \bigvee_7 S^1$.
\end{corollary}

Corollary~\ref{caso4} should be compared to the fact that the configuration space of \emph{unordered} pairs of points in $|K_4|$ has the homotopy type of $\bigvee_4 S^1$ (cf.~\cite[Example~4.5]{MR2171804}.)

\subsection{Morse cochain complex}
The Morse coboundary map $\delta\colon \mu^0(C_m)\to\mu^1(C_m)$ is forced to vanish since $c_0=1$. More interestingly:
\begin{proposition}\label{delta1}
The coboundary $\delta\colon \mu^1(C_m)\to\mu^2(C_m)$ vanishes on the duals of the critical faces of types (k.2) and (k.3) in Proposition~\ref{Wn}. For the duals of the critical faces of type (k.1) we have
\begin{equation}\label{cobory}
\delta\left(\left[\begin{matrix}a&m-1\\b&m\end{matrix}\right]\right)=
\sum\left[\begin{matrix}a&a&x\\y&b&b\end{matrix}\right]-
\sum\left[\begin{matrix}a&a&x\\b&y&y\end{matrix}\right]+
\sum\left[\begin{matrix}x&x&a\\b&y&y\end{matrix}\right]-
\sum\left[\begin{matrix}x&x&a\\y&b&b\end{matrix}\right],
\end{equation}
where all four summands run over all integers $x$ and $y$ that render critical 2-faces. Explicitly, $a<x<m$ in the first and second summations, $x<a$ in the third and fourth summations, $b<y<m$ in the second and third summations, $y<b$ in the first and fourth summations, and $b\neq x\neq y\neq a$ in all four summations.
\end{proposition}

Note that the first two summations in~(\ref{cobory}) are empty when $a=m-1$ (so $b<m-2$).

\begin{proof}
The complete trees of mixed paths $\beta\searrow\,\nearrow\cdots\nearrow\;\searrow\alpha$ from critical 2-dimensional faces~$\beta$ to either critical or collapsible 1-dimensional faces~$\alpha$ are spelled out in Figures~\ref{fig11}--\ref{fig33}. In the figures, we indicate a positive (respectively negative) face with a bold (respectively regular) arrow. Types of pairings involved are indicated using the item names~\emph{(\ref{wn12})--(\ref{wnultimo})} in Proposition~\ref{Wn}. At the end of each branch, we indicate either the type of paring that shows $\alpha$ is collapsible or, if $\alpha$ is critical, the multiplicity with which the path must be accounted for in~(\ref{morseboundary}) and~(\ref{morsecoboundary}).

\smallskip
The first assertion in the proposition follows by observing that, in Figures~\ref{fig11}--\ref{fig33}, there are two mixed paths departing from a fixed critical 2-dimensional face and arriving to a fix critical 1-dimensional face of the form~\emph{(k.2)} or~\emph{(k.3)}. These two mixed paths have opposite multiplicities, so they cancel each other out in~(\ref{morsecoboundary}). For instance, each mixed path from
$\left[
\genfrac{}{}{0pt}{1}{a}{b}
\genfrac{}{}{0pt}{1}{a}{d}
\genfrac{}{}{0pt}{1}{c}{d}
\right]$
to
$\left[
\genfrac{}{}{0pt}{1}{a}{m}
\genfrac{}{}{0pt}{1}{c}{m}
\right]$
in Figure~\ref{fig11} cancels out with the corresponding path in Figure~\ref{fig12}.

\smallskip
To get at~(\ref{cobory}), start by noticing from Figures~\ref{fig11}--\ref{fig33} that there are only four types of mixed paths departing from a given critical 2-dimensional face $\left[
\genfrac{}{}{0pt}{1}{a}{b}
\genfrac{}{}{0pt}{1}{a}{d}
\genfrac{}{}{0pt}{1}{c}{d}
\right]$ and that arrive to some critical 1-dimensional faces of type~\emph{(k.1)}. Namely,
\begin{itemize}
\item there is a mixed path $\left[
\genfrac{}{}{0pt}{1}{a}{b}
\genfrac{}{}{0pt}{1}{a}{d}
\genfrac{}{}{0pt}{1}{c}{d}
\right]\searrow\nearrow\cdots\nearrow\searrow\left[
\genfrac{}{}{0pt}{1}{a}{d}
\genfrac{}{}{0pt}{1}{m-1}{m}
\right]$ having multiplicity $+1$ (Figures~\ref{fig11}, \ref{fig21} and~\ref{fig31});
\item there is a mixed path $\left[
\genfrac{}{}{0pt}{1}{a}{b}
\genfrac{}{}{0pt}{1}{a}{d}
\genfrac{}{}{0pt}{1}{c}{d}
\right]\searrow\nearrow\cdots\nearrow\searrow\left[
\genfrac{}{}{0pt}{1}{a}{b}
\genfrac{}{}{0pt}{1}{m-1}{m}
\right]$ having multiplicity $-1$ (Figures~\ref{fig12}, \ref{fig22} and~\ref{fig32});
\item there is a mixed path $\left[
\genfrac{}{}{0pt}{1}{a}{b}
\genfrac{}{}{0pt}{1}{a}{d}
\genfrac{}{}{0pt}{1}{c}{d}
\right]\searrow\nearrow\cdots\nearrow\searrow\left[
\genfrac{}{}{0pt}{1}{c}{b}
\genfrac{}{}{0pt}{1}{m-1}{m}
\right]$ having multiplicity $+1$ (Figures~\ref{fig12}, \ref{fig22} and~\ref{fig32});
\item there is a mixed path $\left[
\genfrac{}{}{0pt}{1}{a}{b}
\genfrac{}{}{0pt}{1}{a}{d}
\genfrac{}{}{0pt}{1}{c}{d}
\right]\searrow\nearrow\cdots\nearrow\searrow\left[
\genfrac{}{}{0pt}{1}{c}{d}
\genfrac{}{}{0pt}{1}{m-1}{m}
\right]$ having multiplicity $-1$ provided $(c,d)\neq(m-1,m-2)$ (Figures~\ref{fig11}, \ref{fig21} and~\ref{fig31}).
\end{itemize}
Therefore the value of the boundary map $\partial\colon \mu_2(C_m)\to\mu_1(C_m)$ at a critical face
$\left[
\genfrac{}{}{0pt}{1}{a}{b}
\genfrac{}{}{0pt}{1}{a}{d}
\genfrac{}{}{0pt}{1}{c}{d}
\right]$
($\hspace{.3mm}b\neq c<m>d$) with $(c,d)\neq(m-1,m-2)$ is
\begin{equation}\label{delta2}
\partial\left(\left[\begin{matrix}a&a&c\\b&d&d\end{matrix}\right]\right)=
\left[\begin{matrix}a&m-1\\d&m\end{matrix}\right]-
\left[\begin{matrix}a&m-1\\b&m\end{matrix}\right]+
\left[\begin{matrix}c&m-1\\b&m\end{matrix}\right]-
\left[\begin{matrix}c&m-1\\d&m\end{matrix}\right],
\end{equation}
whereas, for $(c,d)=(m-1,m-2)$,
\begin{equation}\label{delta3}
\partial\left(\left[\begin{matrix}a&a&m-1\\b&m-2&m-2\end{matrix}\right]\right)=
\left[\begin{matrix}a&m-1\\m-2&m\end{matrix}\right]-
\left[\begin{matrix}a&m-1\\b&m\end{matrix}\right]+
\left[\begin{matrix}m-1&m-1\\b&m\end{matrix}\right].
\end{equation}
(Note that the expression (\ref{delta2}) is valid when $(c,d)=(m-1,m-2)$ provided the fourth non-critical term
$$
\left[\begin{matrix}m-1&m-1\\m-2&m\end{matrix}\right]
$$
is omitted.) Expression~(\ref{cobory}) then follows by dualizing~(\ref{delta2}) and~(\ref{delta3}).
\end{proof}

\begin{figure}
\centering $\hspace{8.5mm}
\begin{cases} \searrow 
\md acdm \stackrel{(\ref{wn910})}\nearrow 
\mt aacdmm 
\begin{cases} \bp 
\md aadm \stackrel{(\ref{wn3})}\sp 
\mt aa{m-1}dmm 
\begin{cases} \searrow 
\md a{m-1}dm \emph{\ \ $(+)$}
\\ \bp 
\md a{m-1}mm\rule{0mm}{6mm} \emph{\ \ (\ref{wnpenultimo})}
\end{cases}
\\ \bp 
\md acmm \emph{\ \ $(+)$}
\end{cases}
\\ \bp 
\md ccdm \rule{0mm}{10mm}\stackrel{(\ref{wn3})}\sp 
\mt cc{m-1}dmm  
\begin{cases} \searrow 
\md c{m-1}dm \emph{\ \ $(-)$}
\\ \bp 
\md c{m-1}mm \rule{0mm}{6mm} \emph{\ \ (\ref{wnpenultimo})}
\end{cases}
\end{cases}$
\caption{Gradient paths evolving from $\mt aacbdd\pmb{\searrow}\md acdd \stackrel{(\ref{wn45})}{\pmb{\nearrow}}
\mt accddm$ for $b\neq c\leq m-2\geq d$}
\label{fig11}
\end{figure}

\begin{figure}
\centering $\rule{11.5mm}{0mm}
\begin{cases} \bp 
\md acbb \stackrel{(\ref{wn45})}\sp 
\mt accbbm 
\begin{cases} \searrow 
\md acbm \stackrel{(\ref{wn910})}\nearrow 
\mt aacbmm
\begin{cases} \bp 
\md aabm \stackrel{(\ref{wn3})}\sp
\mt aa{m-1}bmm
\begin{cases} \searrow
\md a{m-1}bm \emph{\ \ $(-)$} \\ \bp 
\md a{m-1}mm \emph{\ \ (\ref{wnpenultimo})} \rule{0mm}{6mm} 
\end{cases}
\\ \bp 
\md acmm \rule{0mm}{6mm} \emph{\ \ $(-)$}
\end{cases}
\\ \bp 
\md ccbm \stackrel{(\ref{wn3})}\sp
\mt cc{m-1}bmm
\begin{cases}\searrow 
\md c{m-1}bm \emph{\ \ $(+)$} \rule{0mm}{6mm} \\ \bp
\md c{m-1}mm \emph{\ \ (\ref{wnpenultimo})}\rule{0mm}{6mm}
\end{cases}
\end{cases}
\\ \bp 
\md ccbd \rule{0mm}{10mm}\stackrel{(\ref{wn12})}\sp 
\mt ccmbdd  
\begin{cases} \searrow 
\md cmbd \stackrel{(\ref{wn78})}\nearrow
\mt cmmbbd
\begin{cases} \bp
\md cmbb \stackrel{(\ref{wn6})}\sp 
\mt cmmbb{m-1} 
\begin{cases} \searrow
\md cmb{m-1} \emph{\ \ (\ref{wn145})} \\ \bp 
\md mmb{m-1} \emph{\ \ (\ref{wn145})} \rule{0mm}{6mm}
\end{cases} \\ \bp
\md mmbd \emph{\ \ $(-)$}
 \end{cases}
\\ \bp 
\md cmdd \stackrel{(\ref{wn6})}\sp 
\mt cmmdd{m-1} 
\begin{cases} \searrow
\md cmd{m-1} \emph{\ \ (\ref{wn145})} \rule{0mm}{6mm} \\ \bp
\md mmd{m-1} \emph{\ \ (\ref{wn145})} \rule{0mm}{6mm} 
\end{cases}
\end{cases}
\end{cases}$
\caption{Gradient paths evolving from $\mt aacbdd \searrow
\md acbd \stackrel{(\ref{wn78})}\nearrow 
\mt accbbd$ for $b\neq c\leq m-2\geq d$}
\label{fig12}
\end{figure} 

\begin{figure}
\centering $\rule{9.5mm}{0mm}
\begin{cases} \searrow 
\md ambd \stackrel{(\ref{wn78})}\nearrow 
\mt ammbbd 
\begin{cases} \bp 
\md ambb \stackrel{(\ref{wn6})}\sp 
\mt ammbb{m-1}
\begin{cases} \searrow 
\md amb{m-1} \emph{\ \ (\ref{wn145})}
\\ \bp 
\md mmb{m-1}\rule{0mm}{6mm} \emph{\ \ (\ref{wn145})}
\end{cases}
\\ \bp 
\md mmbd \emph{\ \ $(+)$}
\end{cases}
\\ \bp 
\md amdd \rule{0mm}{10mm}\stackrel{(\ref{wn6})}\sp 
\mt ammdd{m-1}  
\begin{cases} \searrow 
\md amd{m-1} \emph{\ \ (\ref{wn145})}
\\ \bp 
\md mmd{m-1}\rule{0mm}{6mm} \emph{\ \ (\ref{wn145})}
\end{cases}
\end{cases}$
\caption{Gradient paths evolving from $\mt aacbdd \pmb{\searrow}
\md aabd \stackrel{(\ref{wn12})}{\pmb{\nearrow}}
\mt aambdd$ for $b\neq c\leq m-2\geq d$}
\label{fig13}
\end{figure}

\begin{figure}
\centering $\rule{1.5mm}{0mm}
\begin{cases} \searrow
\md a{m-1}dm \emph{\ \ $(+)$} \\ \bp
\md {m-1}{m-1}dm \emph{\ \ $\genfrac{}{}{0pt}{1}{(-),\text{ if }d\,<\,m-2}{(\ref{wnultimo}),\hspace{1.2mm}\text{ if }d\,=\,m-2}$} \rule{0mm}{6mm}
\end{cases}$
\caption{Gradient paths evolving from $\mt aa{m-1}bdd \pmb{\searrow}
\md a{m-1}dd \stackrel{(\ref{wn45})}{\pmb{\nearrow}}
\mt a{m-1}{m-1}ddm $ for $b\neq m-1>d$}
\label{fig21}
\end{figure}

\begin{figure}{\small
\centering $
\begin{cases} \bp 
\md a{m-1}bb \stackrel{(\ref{wn45})}\sp 
\mt a{m-1}{m-1}bbm 
\begin{cases} \searrow 
\md a{m-1}bm  \emph{\ \ $(-)$}
\\ \bp 
\md {m-1}{m-1}bm \emph{\ \ $(+)$} \rule{0mm}{6mm}
\end{cases}
\\ \rule{-.4mm}{10mm} \bp 
\md {m-1}{m-1}bd \rule{-1.3mm}{10mm}\stackrel{(\ref{wn12})}\sp 
\rule{-1mm}{10mm} \mt {m-1}{m-1}mbdd  
\rule{-1.4mm}{0mm}
\begin{cases} \hspace{-.7mm}\searrow \hspace{-1mm}
\md {m-1}mbd \hspace{-.9mm} \stackrel{(\ref{wn78})} \nearrow
\hspace{-.8mm}\mt {m-1}mmbbd \rule{-1.3mm}{0mm}
\begin{cases} \rule{-.5mm}{0mm} \bp \hspace{-.3mm}
\md {m-1}mbb \rule{-1.3mm}{0mm} \stackrel{(\ref{wn6})}\sp \rule{-1mm}{0mm} 
\mt {m-1}mmbb{m-1} \rule{-1.3mm}{0mm}
\begin{cases} \rule{-1.4mm}{0mm}\searrow
\hspace{-.6mm} \md {m-1}mb{m-1} \emph{\ \ (\ref{wn145})} 
\\ \rule{-.6mm}{0mm} \bp \hspace{-.3mm}
\md mmb{m-1} \emph{\ \ (\ref{wn145})} 
\rule{0mm}{6mm}
\end{cases} \\ \bp
\md mmbd \emph{\ \ $(-)$}
 \end{cases}
\\ \bp 
\md {m-1}mdd \stackrel{(\ref{wn6})}\sp 
\mt {m-1}mmdd{m-1} 
\begin{cases} \searrow
\md {m-1}md{m-1} \emph{\ \ (\ref{wn145})} \rule{0mm}{6mm} \\ \bp
\md mmd{m-1} \emph{\ \ (\ref{wn145})} \rule{0mm}{6mm} 
\end{cases}
\end{cases}
\end{cases}$
\caption{Gradient paths evolving from $\mt aa{m-1}bdd \searrow \md a{m-1}bd \stackrel{(\ref{wn78})}\nearrow \mt a{m-1}{m-1}bbd$ for $b\neq m-1>d$}
\label{fig22}
}\end{figure} 

\begin{figure}
\centering $\rule{9.5mm}{0mm}
\begin{cases} \searrow 
\md ambd \stackrel{(\ref{wn78})}\nearrow 
\mt ammbbd 
\begin{cases} \bp 
\md ambb \stackrel{(\ref{wn6})}\sp 
\mt ammbb{m-1}
\begin{cases} \searrow 
\md amb{m-1} \emph{\ \ (\ref{wn145})}
\\ \bp 
\md mmb{m-1}\rule{0mm}{6mm} \emph{\ \ (\ref{wn145})}
\end{cases}
\\ \bp 
\md mmbd \emph{\ \ $(+)$}
\end{cases}
\\ \bp 
\md amdd \rule{0mm}{10mm}\stackrel{(\ref{wn6})}\sp 
\mt ammdd{m-1}  
\begin{cases} \searrow 
\md amd{m-1} \emph{\ \ (\ref{wn145})}
\\ \bp 
\md mmd{m-1}\rule{0mm}{6mm} \emph{\ \ (\ref{wn145})}
\end{cases}
\end{cases}$
\caption{Gradient paths evolving from $\mt aa{m-1}bdd \pmb{\searrow}
\md aabd \stackrel{(\ref{wn12})}{\pmb{\nearrow}}
\mt aambdd$ for $b\neq m-1>d$}
\label{fig23}
\end{figure}

\begin{figure}
\centering $\hspace{8.5mm}
\begin{cases} \searrow 
\md ac{m-1}m \stackrel{(\ref{wn910})}\nearrow 
\mt aac{m-1}mm 
\begin{cases} \bp 
\md aa{m-1}m \stackrel{(\ref{wn3})}\sp 
\mt aa{m-1}{m-1}mm 
\begin{cases} \searrow 
\md a{m-1}{m-1}m \emph{\ \ $(+)$}
\\ \bp 
\md a{m-1}mm\rule{0mm}{6mm} \emph{\ \ (\ref{wnpenultimo})}
\end{cases}
\\ \bp 
\md acmm \emph{\ \ $(+)$}
\end{cases}
\\ \bp 
\md cc{m-1}m \rule{0mm}{10mm}\stackrel{(\ref{wn3})}\sp 
\mt cc{m-1}{m-1}mm  
\begin{cases} \searrow 
\md c{m-1}{m-1}m \emph{\ \ $(-)$}
\\ \bp 
\md c{m-1}mm \rule{0mm}{6mm} \emph{\ \ (\ref{wnpenultimo})}
\end{cases}
\end{cases}$
\caption{Gradient paths evolving from $\mt aacb{m-1}{m-1}\pmb{\searrow}\md ac{m-1}{m-1} \stackrel{(\ref{wn45})}{\pmb{\nearrow}}
\mt acc{m-1}{m-1}m$ for $b\neq c<m-1$}
\label{fig31}
\end{figure}

\begin{figure}
\centering $\rule{11.5mm}{0mm}
\begin{cases} \bp 
\md acbb \stackrel{(\ref{wn45})}\sp 
\mt accbbm 
\begin{cases} \searrow 
\md acbm \stackrel{(\ref{wn910})}\nearrow 
\mt aacbmm
\begin{cases} \bp 
\md aabm \stackrel{(\ref{wn3})}\sp
\mt aa{m-1}bmm
\begin{cases} \searrow
\md a{m-1}bm \emph{\ \ $(-)$} \\ \bp 
\md a{m-1}mm \emph{\ \ (\ref{wnpenultimo})} \rule{0mm}{6mm} 
\end{cases}
\\ \bp 
\md acmm \rule{0mm}{6mm} \emph{\ \ $(-)$}
\end{cases}
\\ \bp 
\md ccbm \stackrel{(\ref{wn3})}\sp
\mt cc{m-1}bmm
\begin{cases}\searrow 
\md c{m-1}bm \emph{\ \ $(+)$} \rule{0mm}{6mm} \\ \bp
\md c{m-1}mm \emph{\ \ (\ref{wnpenultimo})}\rule{0mm}{6mm}
\end{cases}
\end{cases}
\\ \bp 
\md ccb{m-1} \rule{0mm}{10mm}\stackrel{(\ref{wn12})}\sp 
\mt ccmb{m-1}{m-1}  
\begin{cases} \searrow 
\md cmb{m-1} \emph{\ \ (\ref{wn145})} \\ \bp 
\md cm{m-1}{m-1} \emph{\ \ (\ref{wn145})}\rule{0mm}{6mm}
\end{cases}
\end{cases}$
\caption{Gradient paths evolving from $\mt aacb{m-1}{m-1} \searrow
\md acb{m-1} \stackrel{(\ref{wn78})}\nearrow 
\mt accbb{m-1}$ for $b\neq c<m-1$}
\label{fig32}
\end{figure} 

\begin{figure}
\centering $\rule{1.5mm}{0mm}
\begin{cases} \searrow
\md amb{m-1} \emph{\ \ (\ref{wn145})} \\ \bp
\md am{m-1}{m-1} \emph{\ \ (\ref{wn145})} \rule{0mm}{6mm}
\end{cases}$
\caption{Gradient paths evolving from $\mt aacb{m-1}{m-1} \pmb{\searrow}
\md aab{m-1} \stackrel{(\ref{wn12})}{\pmb{\nearrow}}
\mt aamb{m-1}{m-1} $ for $b\neq c<m-1$}
\label{fig33}
\end{figure}

\begin{figure}
\centering
\begin{tikzpicture}[x=1cm,y=1cm]
\node at (.25,4.25) {\tiny $1$};
\node at (.75,4.25) {\tiny $2$};
\node at (1.25,4.25) {\tiny $3$};
\node at (1.75,4.25) {\tiny $4$};
\node at (2.55,4.25) {$\cdots$};
\node at (3.2,4.25) {\tiny $m{-}2$};
\node at (3.8,4.25) {\tiny $m{-}1$};
\node at (0,4.55) {\tiny $\longrightarrow$};
\node at (0,4.68) {\tiny $a$};
\node at (-.56,4) {\tiny $\xdownarrow[1.8mm]$};
\node at (-.66,4) {\tiny $b$};
\node at (-.25,3.75) {\tiny $1$};
\node at (-.25,3.25) {\tiny $2$};
\node at (-.25,2.75) {\tiny $3$};
\node at (-.25,2.05) {$\vdots$};
\node at (.5,2.05) {$(R_1)$};
\node at (3.53,3.725) {\scriptsize $(R_2)$};
\node at (1.25,3.25) {\tiny $(\hspace{-.3mm}R_3\hspace{-.2mm})$};
\node at (2.5,3.24) {\scriptsize $(R_4)$};
\node at (2.5,2.05) {$(R_5)$};
\node at (3.75,1.25) {\tiny $(\hspace{-.3mm}R_6\hspace{-.2mm})$};
\node at (3.25,.25) {\tiny $(\hspace{-.3mm}R_7\hspace{-.2mm})$};
\node at (-.4,1.25) {\tiny $m{-}3$};
\node at (-.4,.75) {\tiny $m{-}2$};
\node at (-.4,.25) {\tiny $m{-}1$};
\draw[very thick](0,0)--(0,4)--(4,4)--(4,1)--(3.5,1)--(3.5,0)--(0,0);
\draw[very thick](1,.5)--(1,3.5)--(3.5,3.5)--(3.5,.5)--(1,.5);
\draw[very thick] (1,3)--(3.5,3);
\draw[very thick] (1.5,3)--(1.5,4);
\draw[very thick] (3,0)--(3,.5)--(3.5,.5)--(3.5,0);
\draw[very thick] (3.5,1.5)--(4,1.5);
\draw[very thick] (3.5,1)--(4,1);
\draw[ultra thin] (0,3.5)--(4,3.5);
\draw[ultra thin] (4,1)--(4,0)--(3.5,0);
\draw[ultra thin] (0,3)--(4,3);
\draw[ultra thin] (0,2.5)--(4,2.5);
\draw[ultra thin] (0,1.5)--(4,1.5);
\draw[ultra thin] (0,1)--(4,1);
\draw[ultra thin] (0,.5)--(4,.5);
\draw[ultra thin] (.5,0)--(.5,4);
\draw[ultra thin] (1,0)--(1,4);
\draw[ultra thin] (1.5,0)--(1.5,4);
\draw[ultra thin] (2,0)--(2,4);
\draw[ultra thin] (3,0)--(3,4);
\draw[ultra thin] (3.5,0)--(3.5,4);
\fill[lightgray](0.0188,3.5)--(.0188,3.981)--(.5,3.981)--(.5,3.5)--(.0188,3.5);
\fill[lightgray](.5,3)--(0.5,3.5)--(.98,3.5)--(.98,3)--(.5,3);
\fill[lightgray](1.0188,2.5)--(1.0188,2.98)--(1.5,2.98)--(1.5,2.5)--(1.0188,2.5);
\fill[lightgray](3,.5)--(3,1)--(3.4795,1)--(3.4795,.52)--(3,.52);
\fill[lightgray](3.52,0)--(3.52,1-.02)--(3.502+.497,1-.02)--(3.502+.497,0)--(3.52,0);
\end{tikzpicture}
\caption{Defining regions for the basis elements $\ma ab$}
\label{newbasisdim1}
\end{figure}

\subsection{Cohomology bases}\label{subsec4.3}
By Corollary~\ref{caso4}, we can assume $m\geq5$ throughout the rest of the paper. We start by identifying (in Corollary~\ref{basedim1} below) an explicit basis for $H^1(\C(|K_m|,2))$, i.e., for the kernel of the Morse coboundary $\delta\colon\mu^1(C_m)\to\mu^2(C_m)$. By Proposition~\ref{delta1}, it is enough to focus on the submodule $\mu^1_0(C_m)$ of $\mu^1(C_m)$ generated by the duals of the basis elements of type~\emph{(k.1)}. Thus, $\mu^1_0(C_m)$ is free on elements $\mb ab$ satisfying $a<m>b\neq a$ and $(a,b)\neq(m-1,m-2)$, where $\mb ab$ stands for the dual of $\md a{m-1}bm$.

\begin{definition}
Consider the elements $\ma ab\in\mu^1_0(C_m)$ defined for $a<m>b\neq a$ and $(a,b)\neq(m-1,m-2)$ according to the following cases (see Figure~\ref{newbasisdim1}):
\begin{enumerate}
\item[$(R_1)$] For $1\leq a\leq 2$, or for $1\leq a\leq m-3$ with $b=m-1$, or for $(a,b)=(3,1)$, $$\ma ab:=\sum_{a\neq j\leq b}\mb aj.$$
\item[$(R_2)$] For $4\leq a\leq m-1$ with $b=1$, or for $a=m-1$ with $1\leq b\leq m-4$, $$\ma ab:=\sum_{b\neq i\leq a}\mb ib.$$
\item[$(R_3)$] $\ma 32:=\mb 32+\mb 31+\mb 23+\mb 21+\mb 13+\mb 12.$
\item[$(R_4)$] For $4\leq a\leq m-2$, $$\ma a2:=\sum_{2\neq i\leq a}\mb i2+\sum_{i\leq a-1}\mb ia.$$
\item[$(R_5)$] For $a,b\in\{3,4,\ldots,m-2\}$, $$\ma ab:=\sum_{j\neq i\leq a\neq j\leq b}\mb ij+\sum_{i\leq a-1}\mb ia.$$
\item[$(R_6)$] $\ma {m-1}{m-3}:=\mb {m-1}{m-3}-\sum_{i,j}\mb ij,$ where the sum runs over $i$ and $j$ with $m-3\neq j\neq i\leq m-2\geq j$.
\item[$(R_7)$] $\ma {m-2}{m-1}:=\mb {m-2}{m-1}-\sum_{i,j}\mb ij,$ where the sum runs over $i$ and $j$ with $m-2\geq j\neq i\leq m-3$.
\end{enumerate}
\end{definition}

Direct inspection of the defining formul\ae\ yields:
\begin{lemma}\label{aux1} The following relations hold under the indicated conditions:
\begin{enumerate}[(i)]
\item $\displaystyle\ma ab-\ma a{b-1}=\sum_{b\neq i\leq a}\mb ib$, provided $3\leq a\leq m-2$ and $4\leq b\leq m-2$ with $a\neq b\neq a+1$.
\item $\displaystyle\ma a3-\ma a2-\ma a1=\sum_{3\neq i\leq a}\mb i3$, provided $4\leq a\leq m-2$.
\item $\displaystyle\ma {b-1}b-\ma {b-1}{b-2}=\sum_{i<b}\mb ib$, provided $5\leq b\leq m-2$.
\item $\displaystyle\ma 34 -\ma 32=\sum_{i\leq 3}\mb i4$, provided $m\geq6$.
\end{enumerate}
\end{lemma}

\begin{proposition}\label{cambiodebasedim1}
The elements $\ma ab$ for $a<m>b\neq a$ and $(a,b)\neq(m-1,m-2)$ yield a basis of $\mu_0^1$.
\end{proposition}
\begin{proof}
In view of the one-to-one correspondence $\mb ab\leftrightarrow\ma ab$, it suffices to check that
\begin{equation}\label{zlc}
\mbox{\emph{each $\mb ab$ is a $\mathbb{Z}$-linear combination of the elements $\ma {a'}{b'}$.}}
\end{equation}
In turn, the fact in~(\ref{zlc}) follows in most cases by a simple recursive argument based on the observation that, in all cases,
\begin{equation}\label{allcases}
\mb ab\,=\,\ma ab\,\,+\!\!\!\sum_{(a',b')\neq(a,b)} (\pm1)\cdot\mb {a'}{b'}.
\end{equation}
Namely, the recursive argument applies for ($R_1$) when $a\leq2$ or $b=1$, for ($R_2$) when $b=1$, and for ($R_3$). The recursive argument also applies in the cases ($R_6$) and ($R_7$), as well as in the remaining instances of ($R_1$) and ($R_2$) provided
\begin{equation}\label{etiprov}
\mbox{\emph{(\ref{zlc}) holds true when $a$ and $b$ fall in cases ($R_4$) and ($R_5$).}}
\end{equation}
Since the cases ($R_4$) and ($R_5$) are empty for $m=5$, we only need to verify~(\ref{etiprov}) assuming $m\geq6$.

Direct computation gives $\mb 42=\ma 42-\ma 34+\ma 31+\ma 23+\ma 13-\ma 12$, while items \emph{(iii)} and \emph{(iv)} of Lemma~\ref{aux1} yield $$\mb a2=\left(\ma a2-\ma {a-1}2\right)+\left(\ma {a-2}{a-1}-\ma {a-2}{a-3} \right)-\left(\ma {a-1}a-\ma {a-1}{a-2} \right) $$ for $5\leq a\leq m-2$, which establishes~(\ref{etiprov}) in the case of~($R_4$). The validity of~(\ref{etiprov}) in the case $a=3$ of ($R_5$) is established in a similar way: Note first that the idea in the recursive argument based on~(\ref{allcases}) works to give~(\ref{etiprov}) for $(a,b)=(3,4)$; then use item~\emph{(i)} in Lemma~\ref{aux1} to get
$$
\mb 3b=\left(\ma 3b-\ma 3{b-1}\right)-\ma 2b+\ma 2{b-1}-\ma 1b+\ma 1{b-1}
$$
for $5\leq b\leq m-2$. The validity of~(\ref{etiprov}) in the case $b=3$ of~($R_5$) is established using the idea in the recursive argument at the beginning of the proof, except that~(\ref{allcases}) is replaced by the expression
$$
\mb a3=\ma a3-\ma a2-\ma a1-\sum_{3\neq i<a}\mb i3
$$
coming from item~\emph{(ii)} in Lemma~\ref{aux1}. Lastly, the validity of~(\ref{etiprov}) in the remaining case $a,b\in\{4,\ldots,m-2\}$ of~($R_5$) is established by the formul\ae
\begin{itemize}
\item $\displaystyle\mb ab=\left(\ma ab-\ma a{b-1}\right)-\left(\ma {a-1}b-\ma {a-1}{b-1}\right)$, when $a+1\neq b\neq a-1$; 
\item $\displaystyle\mb ab=\left(\ma a{a-1}-\ma a{a-2}\right)-\left(\ma {a-2}{a-1}-\ma {a-2}{a-3}\right)$, when $b=a-1$ (so $a\geq5$); 
\item $\displaystyle\mb ab=\left(\ma a{a+1}-\ma a{a-1}\right)-\left(\ma {a-1}{a+1}-\ma {a-1}a\right)$, when $b=a+1$ (so $a\leq m-3$), 
\end{itemize}
which use items~\emph{(i)}, \emph{(iii)} and~\emph{(iv)} in Lemma~\ref{aux1}.
\end{proof}

Recall that the Morse coboundary $\delta\colon\mu^0(C_m)\to\mu^1(C_m)$ is trivial, so that the 1-dimensional cohomology of $\C(|K_m|,2)$ is given by the kernel of $\delta\colon\mu^1(C_m)\to\mu^2(C_m)$.  
\begin{corollary}\label{basedim1}
A basis for $H^1(\C(|K_m|,2))$ is given by
\begin{enumerate}[(1)]
\item the duals of the critical 1-dimensional faces of type~(k.2) and~(k.3) in Proposition~\ref{Wn} and
\item the (already dualized) elements $\ma ab$ satisfying either $a=m-1$, or $b=m-1$, or $(a,b)=(m-2,m-3)$.
\end{enumerate}
\end{corollary}
\begin{proof}
Recall $m\geq5$. A straightforward counting shows that the number of elements in~\emph{(1)} and~\emph{(2)} above is $(m-1)(m-2)$, which is also the first Beti number of $\C(|K_m|,2)$ ---cf.~\cite[Corollary~23]{MR2745669}). Since the homology of $\C(|K_m|,2)$ is torsion-free (\cite[Proposition~2]{MR2745669}), Proposition~\ref{cambiodebasedim1} implies that the proof will be complete once it is checked that $\delta\colon\mu^1(C_m)\to\mu^2(C_m)$ vanishes on each of the elements in~\emph{(1)} and~\emph{(2)}. Indeed, the latter fact actually implies that $\delta$ is injective on the submodule generated by the basis elements $\ma ab$ not included in~\emph{(2)}.

The vanishing of $\delta$ on the elements in~\emph{(1)} comes directly from Proposition~\ref{delta1}, whereas the vanishing of $\delta$ on the elements in~\emph{(2)} is verified by direct calculation using the expression of $\delta$ in Proposition~\ref{delta1}. The arithmetical manipulations needed are illustrated next in a representative case, namely, that of $\ma {m-1}{m-3}$.

Use Proposition~\ref{delta1} and the defining formula ($R_6$) to get
\begin{align*}
\delta\ma{m-1}{m-3}=&\sum_{x,y}\mt xx{m-1}{m-3}yy-\sum_{x,y}\mt xx{m-1}y{m-3}{m-3}-\\&\sum_{m-3\neq j\neq i\leq m-2\geq j}\left(\sum_{x,y}\mt iixyjj-\sum_{x,y}\mt iixjyy+\sum_{x,y}\mt xxijyy-\sum_{x,y}\mt xxiyjj\right).
\end{align*}
The summands with $y=m-1$ in the second inner summation cancel out the corresponding ones in the third inner summation. (The corresponding fact for $y<m-1$, dealt with below, is more subtle since $i\leq m-2$ in the third inner summation, while $x\leq m-1$ in the second inner summation.) Noticing in addition that $y=m-2$ is forced in the first outer summation, we then get
\begin{align*}
\delta\ma{m-1}{m-3}=&\sum_{x}\mt xx{m-1}{m-3}{m-2}{m-2}-\sum_{x,y}\mt xx{m-1}y{m-3}{m-3}-\\&\sum_{m-3\neq j\neq i\leq m-2\geq j}\left(\sum_{x,y}\mt iixyjj-\sum_{\genfrac{}{}{0pt}{1}{x}{y\leq m-2}}\mt iixjyy+\sum_{\genfrac{}{}{0pt}{1}{x}{y\leq m-2}}\mt xxijyy-\sum_{x,y}\mt xxiyjj\right).
\end{align*}
In the last expression, the summands with $x\leq m-2$ in the second inner summation cancel out the third inner summation, so
\begin{align*}
\delta\ma{m-1}{m-3}=&\sum_{x}\mt xx{m-1}{m-3}{m-2}{m-2}-\sum_{x,y}\mt xx{m-1}y{m-3}{m-3}-\\&\sum_{m-3\neq j\neq i\leq m-2\geq j}\left(\sum_{x,y}\mt iixyjj-\sum_{y\leq m-2}\mt ii{m-1}jyy-\sum_{x,y}\mt xxiyjj\right).
\end{align*}
Likewise, in the last expression, summands with $x\leq m-2$ in the first inner summation cancel out the third inner summation, so
\begin{align*}
\delta\ma{m-1}{m-3}=\sum_{x}\mt xx{m-1}{m-3}{m-2}{m-2}-\sum_{x,y}\mt xx{m-1}y{m-3}{m-3}\;\;-\hspace{-2mm}\sum_{m-3\neq j\neq i\leq m-2\geq j}\left(\sum_{y}\mt ii{m-1}yjj-\sum_{y\leq m-2}\mt ii{m-1}jyy\right).
\end{align*}
Lastly, merge the first (second, respectively) outer summation and the second (first, respectively) inner summation in the last expression to get
\begin{align*}
\delta\ma{m-1}{m-3}\;\,=\sum_{j\neq i\leq m-2\geq y}\mt ii{m-1}jyy\;\;-\sum_{j\neq i\leq m-2\geq j}\mt ii{m-1}yjj\;\,=\;\,0,
\end{align*}
as asserted.
\end{proof}

We next identify (in Corollary~\ref{bd2} below) an explicit basis for $H^2(\C(|K_m|,2))$, i.e., for the cokernel of the Morse coboundary $\delta\colon\mu^1(C_m)\to\mu^2(C_m)$. In what follows, the conditions $1\leq a<c<m$, $1\leq b<d<m$ and $c\neq d\neq a\neq b\neq c<m>d$ for critical faces $\mt aacbdd$ identified in item~\emph{(l)} of Proposition~\ref{Wn} will be implicit (and generally omitted). Let~$\mathcal{C}$ be the collection of the critical faces $\mt aacbdd$ of one of the following four types
\begin{align}
\mt 11c2dd\!, &\text{ with } c,d\in\{3,4,\ldots,m-1\}; \label{typ1} \\
\mt 1123dd\!, &\text{ with } d\in\{4,5,\ldots,m-1\}; \label{typ2} \\ 
\mt 22c133\!, &\text{ with } c\in\{4,5,\ldots,m-1\}; \label{typ3} \\ 
\mt 223144\!, \label{typ4}
\end{align}
and let $\mathcal{B}$ stand for the collection of all other critical faces $\mt aacbdd$. The following change of basis is used to show that the duals of critical faces in $\mathcal{B}$ form a basis of $H^2(\C(|K_m|,2))$:

\begin{definition}\label{cambiobasedim2}
For each $\mt aacbdd\in\mathcal{B}$, consider the element $\nb aacbdd\in\mu_2(C_m)$ defined through:
\begin{enumerate}
\item Case $a=1$ and $c\geq3$ with $\mt aacbdd$ not fitting in~$(\ref{typ1})$,
\begin{enumerate}
\item $\nb 11cbdd:=\mt 11cbdd - \mt 11c2dd + \mt 11c2bb$, for $b\geq3$;
\end{enumerate}
\item Case $a=1$ and $c=2$ with $\mt aacbdd$ not fitting in~$(\ref{typ2})$,
\begin{enumerate}\addtocounter{enumii}{1}
\item $\nb 112bdd:=\mt 112bdd - \mt 1123dd + \mt 1123bb$, for $b\geq4$;
\end{enumerate}
\item Case $a=2$, $b=1$ and $c\geq4$ with $\mt aacbdd$ not fitting in~$(\ref{typ3})$,
\begin{enumerate}\addtocounter{enumii}{2}
\item $\nb 22c1dd:=\mt 22c1dd - \mt 22c133 + \mt 1123dd - \mt 11c2dd + \mt 11c233$, for $d\geq4$;
\end{enumerate}
\item Case $a=2$, $b=1$ and $c=3$ with $\mt aacbdd$ not fitting in~$(\ref{typ4})$,
\begin{enumerate}\addtocounter{enumii}{3}
\item $\nb 2231dd:=\mt 2231dd - \mt 223144 + \mt 1123dd - \mt 112344 + \mt 113244 - \mt 1132dd$, for $d\geq5$;
\end{enumerate}
\item Case $a=2$ and $b\geq3$,
\begin{enumerate}\addtocounter{enumii}{4}
\item $\nb 22c3dd:=\mt 22c3dd + \mt 1123dd - \mt 11c2dd + \mt 11c233$;
\item $\nb 22cbdd:=\mt 22cbdd - \mt 1123bb + \mt 1123dd - \mt 11c2dd + \mt 11c2bb$, for $b\geq4$;
\end{enumerate}
\item Case $a=3$ and $b=1$,
\begin{enumerate}\addtocounter{enumii}{6}
\item $\nb 33c122:=\mt 33c122 + \mt 223144 - \mt 22c133 + \mt 112344 - \mt 113244 + \mt 11c233$;
\item $\nb 33c144:=\mt 33c144 + \mt 223144 - \mt 22c133 + \mt 112344 - \mt 11c244 + \mt 11c233$;
\item $\nb 33c1dd:=\mt 33c1dd + \mt 223144 - \mt 22c133 + \mt 112344 - \mt 11c2dd + \mt 11c233 - \mt 113244 + \mt 1132dd$, for $d\geq5$;
\end{enumerate}
\item Case $a=3$ and $b\geq2$,
\begin{enumerate}\addtocounter{enumii}{9}
\item $\nb 33c2dd:=\mt 33c2dd - \mt 11c2dd + \mt 1132dd$;
\item $\nb 33cbdd:=\mt 33cbdd - \mt 11c2dd + \mt 11c2bb - \mt 1132bb + \mt 1132dd$, for $b\geq4$;
\end{enumerate}
\item Case $a\geq4$,
\begin{enumerate}\addtocounter{enumii}{11}
\item $\nb aac122:=\mt aac122 - \mt 11a233 + \mt 11c233  - \mt 22c133 + \mt 22a133 $;
\item $\nb aac133:=\mt aac133 - \mt 22c133 + \mt 22a133$;
\item $\nb aac1dd:=\mt aac1dd - \mt 22c133 + \mt 22a133 - \mt 11c2dd + \mt 11c233 - \mt 11a233 + \mt 11a2dd$, for $d\geq4$;
\item $\nb aac2dd:=\mt aac2dd - \mt 11c2dd + \mt 11a2dd$;
\item $\nb aacbdd:=\mt aacbdd - \mt 11c2dd + \mt 11c2bb - \mt 11a2bb + \mt 11a2dd$, for $b\geq3$.
\end{enumerate}
\end{enumerate}
\end{definition}

Direct inspection shows that
\begin{equation}\label{dirins}
\mbox{\it each $\nb aacbdd - \mt aacbdd$ is a linear combination of basis elements in $\mathcal{C}$.}
\end{equation}
Therefore $\mathcal{B}'\cup\mathcal{C}$ is a new basis of $\mu_2(C_m)$, where $\mathcal{B}'$ stands for the collection of elements $\nb aacbdd$. Furthermore, routine verifications using~(\ref{delta2}) and~(\ref{delta3}) show that $\mathcal{B}'$ lies in the kernel of $\partial\colon\mu_2(C_m)\to\mu_1(C_m)$. In fact:

\begin{lemma}\label{eleskernelhomologia}
$\mathcal{B}'$ is a basis of the kernel of the Morse boundary map $\partial\colon \mu_2(C_m)\to \mu_1(C_m)$.
\end{lemma}
\begin{proof}
The argument is parallel to that in the proof of Corollary~\ref{basedim1}. Namely, by direct counting, the cardinality of $\mathcal{C}$ is $|\mathcal{C}|=m^2-5m+5$. In view of~(\ref{rangos}), this leads to
$$
|\mathcal{B}'|=|\mathcal{B}|=\frac{m(m-2)(m-3)(m-5)}{4}+1,
$$
which is the second Betti number of $\C(|K_m|,2)$ ---cf.~\cite[Corollary~23]{MR2745669}. Consequently, $\partial\colon \mu_2(C_m)\to \mu_1(C_m)$ is forced to be injective on the submodule spanned by $\mathcal{C}$ and, in particular, $\mathcal{B}'$ spans (and is thus a basis of) ker$(\partial\colon \mu_2(C_m)\to \mu_1(C_m))$.
\end{proof}

The two sequences
\begin{align}\label{notortion}\begin{gathered}
0\longrightarrow H_2(\C(|K_m|,2))\stackrel{\iota}{\longhookrightarrow}\mu_2(C_m)\stackrel{\partial}{\longrightarrow}\mu_1(C_m) \\
0\longleftarrow H^2(\C(|K_m|,2))\stackrel{\;\,\iota^*}{\longleftarrow}\mu^2(C_m)\stackrel{\delta}{\longleftarrow}\mu^1(C_m)
\end{gathered}\end{align}
are exact; the first one by definition, and the second one, which is the dual of the first one, because $H_1(\C(|K_m|,2))$ is torsion-free. Thus, the cohomology class represented by an element $e\in\mu^2(C_m)$ is given by the $\iota^*$-image of $e$. Such an interpretation of cohomology classes is used in the proof of:

\begin{corollary}\label{bd2}
A basis of $H^2(\C(|K_m|,2))$ is given by the classes represented by the duals of the critical faces in~$\mathcal{B}$. Furthermore, the expression (as a linear combination of basis elements) of the cohomology class represented by the dual of a critical face in $\mathcal{C}$ is obtained from equations $(E_i)$, $1\leq i\leq 6$, which are congruences modulo the image of $\delta\colon\mu^1(C_m)\to\mu^2(C_m)$.
\begin{itemize}
\item[$(E_1)$] $\displaystyle\;\;
\sum_{x} \mtt xxc233 -
\sum_{x} \mtt ccx233 \hspace{1mm} \equiv \hspace{-1mm} 
\sum_{\genfrac{}{}{0pt}{1}{y\neq 3}{z\in \{1,3\}}} 
\hspace{-2.5mm} \left(\hspace{1mm}
\sum_{x} \mtt xxczyy -
\sum_{x} \mtt ccxzyy \right)$, for $c>3$.

\item[$(E_2)$] $\displaystyle\;\;
\sum_{x} \mtt xx3244 -
\sum_{x} \mtt 33x244 \hspace{1mm} \equiv \hspace{-1mm} 
\sum_{\genfrac{}{}{0pt}{1}{y\neq 4}{z\in \{1,4\}}} 
\hspace{-2.5mm} \left(\hspace{1mm}
\sum_{x} \mtt xx3zyy -
\sum_{x} \mtt 33xzyy \right)$.

\item[$(E_3)$] $\displaystyle\;\;
\sum_{x,y} \mtt xxcydd -
\sum_{x,y} \mtt ccxydd \hspace{1mm} \equiv \hspace{1.5mm} 
\sum_{x,y} \mtt xxcdyy -
\sum_{x,y} \mtt ccxdyy$, for $3\leq c\neq d\geq4$ with $(c,d)\neq(3,4)$.

\item[$(E_4)$] $\displaystyle\;\;
\sum_{x,y} \mt xx2y44 \;-\hspace{-1.2mm}\sum_{\mbox{\tiny$(x,\!y){\neq}(3,\!1)$}} \!\mt 22xy44 \;\equiv\;
\left(\sum_{x,y} \mt xx24yy - \sum_{x,y} \mt 22x4yy\right)-
\left( \sum_{\mbox{\tiny$\genfrac{}{}{0pt}{1}{y>4}{x}$}} \mt xx31yy - \sum_{x,y} \mt 33x1yy\right).$

\item[$(E_5)$] $\displaystyle\;\;
\sum_{x,y} \mt xx2ydd -\sum_{x,y} \!\mt 22xydd \;\equiv\;
\sum_{x,y} \mt xx2dyy - \sum_{x,y} \mt 22xdyy$, for $d>4$.

\item[$(E_6)$] $\displaystyle\;\;
\sum_{x,y} \mt xxc1yy \;\equiv\;
\sum_{x,y} \mt ccx1yy$, for $c>2$.
\end{itemize}
\end{corollary}
Formally, $(E_3)$ and $(E_5)$ agree, while $(E_2)$ and $(E_1)$ are nearly identical. 
\begin{proof}
The first assertion follows from~(\ref{dirins}) and~(\ref{notortion}). For the second assertion, start by noting that the listed congruences are obtained by dualizing the 16 formul\ae\ (in Definition~\ref{cambiobasedim2}) that describe the inclusion~$\iota$. Indeed, the validness of the congruences is obtained by a straightforward verification (left as an exercise for the reader) of the fact that both sides of each congruence evaluate the same at each basis element $\nb aacbdd$. Furthermore, direct inspection shows that, in each equation $(E_i)$, there is a single summand~$s_i$ (spelled out in Figure~\ref{esesi}) that fails to come from~$\mathcal{B}$. Therefore $(E_i)$ can be thought of as expressing the cohomology class represented by~$s_i$ as a $\mathbb{Z}$-linear combination of basis elements. The second assertion of the corollary then follows by observing, from Figure~\ref{esesi}, that each element in $\mathcal{C}$ arises as one, and only one, of the special summands $s_i$.
\end{proof}

\begin{figure}
\centering\begin{tabular}{r|ccccccc}
$i$ & $1$ & $2$ & $3$ & $4$ & $5$ & $6$ & $6$ \\
$s_i$ & $\rule{0mm}{5mm}\mt 11c233$ & \;\,$\mt 113244$ & $\mt 11c2dd$ & $\mt 112344$ & \;\,$\mt 1123dd$ & \;\,$\mt 22c133$ & \;\,$\mt 223144$ \\
\rule{0mm}{5mm}type & (\ref{typ1}) & (\ref{typ1}) & (\ref{typ1}) & (\ref{typ2}) & (\ref{typ2}) & (\ref{typ3}) & (\ref{typ4}) \\
\rule{0mm}{5mm}restrictions & $c>3$ &  & $\genfrac{}{}{0pt}{1}{3\leq c\neq d\geq4}{(c,d)\neq(3,4)}$ & & $d>4$ & $c>3$ & $$
\end{tabular}
\caption{Elements coming from $\mathcal{C}$ in the congruences $(E_i)$ of Corollary~\ref{bd2}.}
\label{esesi}
\end{figure}

\subsection{Cohomology ring}
In previous sections we have described explicit cocycles in $\mu^*(C_m)$ representing basis elements in cohomology. We now make use of~(\ref{cupprod}) in order to assess the corresponding cup products ---both at the critical cochain level, as well as at the homology level. Since cup products in $C^*(C_m)$ are elementary (see Remark~\ref{eleprocup} below), the bulk of the work amounts to giving a (suitable) description of the cochain maps $\overline{\Phi}\colon\mu^*(C_m)\to C^*(C_m)$ and $\underline{\Phi}\colon C^*(C_m)\to \mu^*(C_m)$.

\begin{remark}\label{eleprocup}{\em
Recall that basis elements in $C^1(C_m)$ are given by the dualized 1-dimensional faces $\md acbd$. (As in earlier parts of the paper, upper stars for dualized elements are omitted, and arithmetical restrictions among the numbers assembling critical faces are ussually not written down.) From the usual formula for cup products in the simplicial setting, we see that the only non-trivial products in $C^*(C_m)$ have the form 
\begin{equation}\label{pncero}
\md aabd\smile\md acdd=\mt aacbdd\qquad \text{or}\qquad \md acbb \smile\md ccbd=\mt accbbd.
\end{equation}
(So every 2-face is uniquely a product of two 1-faces.) In particular, for the purposes of applying~(\ref{cupprod}), all basis elements $\md acbd$ with $a<c$ and $b<d$ can be ignored in the expression for $\overline{\Phi}$.
}\end{remark}

\begin{proposition}\label{fiarriba}
The values of the cochain map $\overline{\Phi}\colon\mu^1(C_m)\to C^1(C_m)$ on the basis elements (k.1)--(k.3) of Proposition~\ref{Wn} satisfy the following congruences modulo basis elements $\md acbd$ with $a<c$ and $b<d$$\hspace{.5mm}:$
\begin{enumerate}[(i)]
\item $\overline{\Phi}\left(\md a{m-1}bm\right)\equiv\md aabm+\sum\md axbb-\sum\md yabb$, where the first summation runs over integers $x\in\{1,\ldots,m-1\}$ with $a<x\neq b$, and the second summation runs over integers $y\in\{1,\ldots,m-1\}$ with $b\neq y<a$;
\item $\overline{\Phi}\left(\md mmbd\right)\equiv\sum\md xxbd$, where the summation runs over integers $x\in\{1,\ldots,m\}$ with $b\neq x\neq d$;
\item $\overline{\Phi}\left(\md acmm\right)\equiv\sum\md acyy$, where the summation runs over integers $y\in\{1,\ldots,m\}$ with $a\neq y\neq c$,
\end{enumerate}
\end{proposition}
\begin{proof}
The congruences follow from~(\ref{quasi-isomorphisms}) and from direct inspection of Figures~\ref{fige3}--\ref{fige2}, where we spell out the complete trees of gradient paths landing on critical 1-dimensional faces. Here we follow the notational conventions used in Figures~\ref{fig11}--\ref{fig33}, except that we now keep track of relevant numerical restrictions and, at the start of each path, we indicate the reason that prevents the path from pulling back one further step.
\end{proof}

\begin{figure}
\centering $
\stackrel{\raisebox{1mm}{\tiny$(b{<}m{-}2)$}}{\md {m-1}{m-1}bm} \lbp 
\stackrel{\raisebox{1mm}{\tiny$(b{\neq}y{<}m-1)$}}{\mt y{m-1}{m-1}bbm} 
\stackrel{(\ref{wn45})} \lsp \md y{m-1}bb 
\begin{cases}
\lbp \stackrel{\mbox{\tiny $(b{<}f{<}m{-}1)$}}{\mt y{m-1}{m-1}bbf} \stackrel{(\ref{wn78})}\nwarrow
\md y{m-1}bf \swarrow \stackrel{\mbox{\tiny$(y{\neq}f)$}}{\mt yy{m-1}bff} \;\;\text{(critical)}
\\
\lbp \stackrel{\raisebox{1mm}{\tiny $(y{\neq}z{<}b)$}}{\mt yy{m-1}zbb}\rule{0mm}{9mm}
 \text{\;\;(critical)}
\end{cases}
$
\caption{Gradient paths landing on a critical cell $\md {m-1}{m-1}bm$ of type \emph{(k.1)} with $b<m-2$}
\label{fige3}
\end{figure}

\begin{figure}
\centering $
\md a{m-1}bm \!
\begin{cases}\!
\swarrow \hspace{-1mm}\mt aa{m-1}bmm\! \stackrel{(\ref{wn3})}\lsp
\!\md aabm \!
\begin{cases}\!
\lbp \stackrel{\raisebox{1mm}{\tiny$(b{\neq}y{<}a)$}}{\mt yaabbm} \stackrel{(\ref{wn45})}\lsp
\md yabb\!
\begin{cases}\!
\lbp\hspace{-3mm}
\stackrel{\raisebox{1mm}{\tiny$(y{\neq}w{<}b, w{\neq}a)$}}{\mt yyawbb} \text{(critical)} \\\!
\lbp\rule{0mm}{9mm}\stackrel{\raisebox{1mm}{\tiny$(a{<}b)$}}{\mt yyaabb} \stackrel{(\ref{wn910})}\nwarrow
\md yaab \;\;\text{(free)} \\ 
\lbp \hspace{-3mm}\rule{0mm}{9mm}\stackrel{\raisebox{1mm}{\tiny$(b{<}z{<}m, z{\neq}a)\,$}}{\mt yaabbz} \stackrel{(\ref{wn78})}\nwarrow
\md yabz 
\swarrow \stackrel{\raisebox{1mm}{\tiny$(y{\neq}z)$}}{\mt yyabzz} \;\;\text{(critical)}
\end{cases}
\\ \!\lbp \!\!\!\stackrel{\raisebox{1mm}{\tiny$(a{<}x{<}m{-}1)$}}{\mt aaxbmm} \stackrel{(\ref{wn910})}{\nwarrow}
\!\!\md axbm \!\!\swarrow \stackrel{\raisebox{1mm}{\tiny$(x{\neq}b)$}}{\mt axxbbm} \stackrel{(\ref{wn45})}\lsp 
\!\!\md axbb \!\rule{-.5mm}{19mm}
\begin{cases}\!
\lbp \hspace{-3.5mm}\stackrel{\raisebox{1mm}{\tiny$\!(b{<}r{<}m,r{\neq}x)$}}{\mt axxbbr} 
\stackrel{(\ref{wn78})}\nwarrow\!\! \md axbr \!
\swarrow \stackrel{\raisebox{1mm}{\tiny$(a{\neq}r)$}}{\mt aaxbrr} \;\text{(critical)} \\ \hspace{-.6mm}
\lbp \hspace{-3mm}
\rule{0mm}{9mm}\stackrel{\raisebox{1mm}{\tiny$(a{\neq}s{<}b,s{\neq}x)$}}{\mt aaxsbb} \;\text{(critical)} \\
\lbp \hspace{-.5mm}\rule{0mm}{9mm}\stackrel{\raisebox{1mm}{\tiny$(x{<}b)$}}{\mt aaxxbb} \stackrel{(\ref{wn910})}\nwarrow
\md axxb \;\;\text{(free)}
\end{cases}
\end{cases}
\\\!
\swarrow\!\rule{0mm}{14mm} \stackrel{\raisebox{1mm}{\tiny$(b{<}m{-}1)$}}{\mt a{m-1}{m-1}bbm} \stackrel{(\ref{wn45})}\lsp
\md a{m-1}bb
\begin{cases}
\lbp \stackrel{\raisebox{1mm}{\tiny$(b{<}x{<}m{-}1)$}}{\mt a{m-1}{m-1}bbx} \stackrel{(\ref{wn78})}\nwarrow\!
\md a{m-1}bx 
\swarrow \stackrel{\raisebox{1mm}{\tiny$(x{\neq}a)$}}{\mt aa{m-1}bxx} \;\;\text{(critical)}
\\
\rule{0mm}{10mm}\lbp \stackrel{\raisebox{.9mm}{\tiny$(a{\neq}y{<}b)$}}{\mt aa{m-1}ybb} \;\;\text{(critical)}
\end{cases}
 \end{cases}
$
\caption{Gradient paths landing on a critical cell $\md a{m-1}bm$ of type \emph{(k.1)} with $a<m-1\geq b$}
\label{fige4}
\end{figure}

\begin{figure}
\centering $
\stackrel{\raisebox{1mm}{\tiny$(d{<}m{-}1)$}}{\md mmbd} \lbp 
\stackrel{\raisebox{1mm}{\tiny$(b{\neq}x{<}m)$}}{\mt xmmbbd} \stackrel{(\ref{wn78})} \nwarrow
\md xmbd \swarrow \stackrel{\raisebox{1mm}{\tiny$(x{\neq}d)$}}{\mt xxmbdd} \stackrel{(\ref{wn12})} \lsp
\md xxbd 
\begin{cases}
\lbp \hspace{-4.5mm}\stackrel{\raisebox{1mm}{\tiny$(x{<}y{<}m,b{\neq}y{\neq}d)$}}{\mt xxybdd} 
\text{(critical)} \\
\lbp \rule{0mm}{10mm}\stackrel{\raisebox{1mm}{\tiny$(x{<}b)$}}{\mt xxbbdd} \stackrel{(\ref{wn910})} \nwarrow
\md xbbd \;\;\text{(free)}
\\
\lbp \rule{0mm}{10mm} \stackrel{\raisebox{1mm}{\tiny$(b{\neq}z{<}x)$}}{\mt zxxbbd} \stackrel{(\ref{wn78})} \nwarrow
\md zxbd \swarrow \stackrel{\raisebox{1mm}{\tiny$(d{\neq}z)$}}{\mt zzxbdd} \;\;\text{(critical)}
\end{cases}
$
\caption{Gradient paths landing on a critical cell $\md mmbd$ of type \emph{(k.2)}}
\label{fige1}
\end{figure}

\begin{figure}
\centering $
\stackrel{\raisebox{1mm}{\tiny$(c{<}m{-}1)$}}{\md acmm} \lbp 
\stackrel{\raisebox{1mm}{\tiny$(a{\neq}y{<}m)$}}{\mt aacymm} \stackrel{(\ref{wn910})} \nwarrow
\md acym \swarrow \stackrel{\raisebox{1mm}{\tiny$(c{\neq}y)$}}{\mt accyym} \stackrel{(\ref{wn45})} \lsp
\md acyy 
\begin{cases}
\lbp \hspace{-3mm}\stackrel{\raisebox{1mm}{\tiny$(y{<}z{<}m,z{\neq}c)$}}{\mt accyyz} \stackrel{(\ref{wn78})} \nwarrow
\md acyz \swarrow \stackrel{\raisebox{1mm}{\tiny$(a{\neq}z)$}}{\mt aacyzz} \;\;\text{(critical)}
\\
\lbp \hspace{-2.4mm}\rule{0mm}{10mm}\stackrel{\raisebox{1mm}{\tiny$(a{\neq}z{<}y,z{\neq}c)$}}{\mt aaczyy}  \text{(critical)}
\\
\lbp \rule{0mm}{10mm}\stackrel{\raisebox{1mm}{\tiny$(c{<}y)$}}{\mt aaccyy} \stackrel{(\ref{wn910})} \nwarrow
\md accy \;\; \text{(free)}
\end{cases}
$
\caption{Gradient paths landing on a critical cell $\md acmm$ of type \emph{(k.3)}}
\label{fige2}
\end{figure}

\begin{proposition}\label{fiabajo}
For critical 1-faces $x$ and $y$, the product $\overline{\Phi}(x)\smile\overline{\Phi}(y)$ appearing in~(\ref{cupprod}) is a linear combination $\sum\pm c$ of dualized 2-faces $c$ each of which has one of the following forms:
\begin{enumerate}[(i)]
\item $\mt aacbdd$ for $a<c<m>d>b$ and $b\neq a\neq d\neq c$, with trivial $\underline{\Phi}$-image unless $b\neq c$, in which case
$$\underline{\Phi}\left(\mt aacbdd\right)=\mt aacbdd.$$
\item $\mt accbbd$ for $a<c\leq m-1>d>b$ and $a\neq b\neq c\neq d$, with trivial $\underline{\Phi}$-image unless $a\neq d$, in which case
$$\underline{\Phi}\left(\mt accbbd\right)=-\mt aacbdd.$$
\item $\mt aacbmm$ for $a<c<m-1\geq b$ and $b\neq a\neq m\neq c$, with trivial $\underline{\Phi}$-image unless $b\neq c$, in which case
$$\underline{\Phi}\left(\mt aacbmm\right)=
\displaystyle\sum_{\genfrac{}{}{0pt}{1}{y<b}{a\neq y\neq c}}\mt aacybb\;\,-
\sum_{\genfrac{}{}{0pt}{1}{b<x<m}{a\neq x\neq c}}\mt aacbxx.
$$
\item $\mt accbbm$ for $a<c\leq m-1\geq b$, $a\neq b\neq c$ and either $c<m-1$ or $c=m-1>b+1$, 
with $\underline{\Phi}$-image $$\underline{\Phi}\left(\mt accbbm\right)\;=
\displaystyle \sum_{\genfrac{}{}{0pt}{1}{b<x<m}{a\neq x\neq c}}\mt aacbxx
\;\,-\sum_{\genfrac{}{}{0pt}{1}{y<b}{a\neq y\neq c}}\mt aacybb.$$
\end{enumerate}
\end{proposition}
\begin{proof}
By~(\ref{pncero}), the only 1-faces in the expression of $\overline{\Phi}(\delta)$ that can lead to a summand $\pm\mt rrtsuu$ in the product $\overline{\Phi}(\gamma)\smile\overline{\Phi}(\delta)$ have the form $\pm\md rtuu$. From the expressions of $\overline{\Phi}$ in Proposition~\ref{fiarriba}, this can hold only with $t<m$ and in fact $t<m-1$ whenever $u=m$, in view of the form of the basis elements of type \emph{(k.3)}. So $\mt rrtsuu$ fits either \emph{(i)} or \emph{(iii)}. Likewise, the only 1-faces in the expression of $\overline{\Phi}(\delta)$ that can lead to a summand $\pm\mt rttssu$ in $\overline{\Phi}(\gamma)\smile\overline{\Phi}(\delta)$ have the form $\pm\md ttsu$, which can hold only under one of the following conditions:
\begin{itemize}
\item $u=m\,$ and $\,t\leq m-1$, as well as $s<m-2$ if $t=m-1$ (recall the form of basis elements of type \emph{(k.1)});
\item $u<m-1$ (recall the form of basis elements of type \emph{(k.2)}).
\end{itemize}
In the former possibility, $\mt rttssu$ fits \emph{(iv)}. In the latter possibility, $\mt rttssu$ fits \emph{(ii)} unless $t=m$, in which case
\begin{equation}\label{inspeccion}
\mbox{\em the expression of $\overline{\Phi}(\gamma)$ should include a summand of the form $\pm\md rmss$.}
\end{equation}
But inspection of the expressions of $\overline{\Phi}$ in Proposition~\ref{fiarriba} rules out~(\ref{inspeccion}).

Lastly, the four asserted expressions for the cochain map $\underline{\Phi}$ follow from~(\ref{quasi-isomorphisms}) and the analysis of gradient paths in Figures~\ref{fig11}--\ref{fig33}.
\end{proof}

The Morse theoretic cup product $\stackrel{\mu}\smile$ in~(\ref{cupprod}) is fully determined by Propositions~\ref{fiarriba} and~\ref{fiabajo} together~(\ref{pncero}). The resulting description is spelled out in the Ph.D.~thesis of the first named author. Here we only remark that, in terms of our explicit cocyle representatives for basis elements (Corollaries~\ref{basedim1} and~\ref{bd2}), the cohomological cup product can simply be read off from~(\ref{cupprod}) using the full power of Corollary~\ref{bd2}. Indeed, as explained in the proof of the latter corollary, congruences ($E_1$)--($E_6$) allow us to identify the cohomology class of any linear combination of critical 2-faces arising from~(\ref{cupprod}). This renders an explicit, complete and fully computer-implementable description of the cohomology ring $H^*(\C(|K_m|,2))$.

\begin{example}\label{sigma6}{\em
For $m=5$, Corollaries~\ref{basedim1} and~\ref{bd2} render the following cocycles representing a graded basis for $H^*(\C(|K_m|,2))$. In dimension 2 there is the single cocycle $\mt 334122$ while, in dimension 1, there are the twelf cocycles
\begin{align*}
&\delta_{12}:=\md 5512, \hspace{64mm} \upsilon_{12}:=\md 1255, \\
&\delta_{13}:=\md 5513, \hspace{64mm} \upsilon_{13}:=\md 1355, \\
&\delta_{23}:=\md 5523, \hspace{64mm} \upsilon_{23}:=\md 2355,\\
&\lambda_{14}:=\md 1445+ \md 1435+\md 1425,\hspace{40mm} \lambda_{24}:=\md 2445+ \md 2435+\md 2415, \\
&\lambda_{41}:=\md 4415+ \md 3415+\md 2415, \hspace{40mm} \lambda_{42}:=\md 4425- \md 3415-\md 2435-\md 2415-\md 1435, \\
& \lambda_{32}:=\md 3425+ \md 3415+\md 2435+\md 2415+\md 1435+\md 1425, \hspace{5mm}
\lambda_{34}:=\md 3445- \md 2435-\md 2415 -\md 1435-\md 1425.
\end{align*}
The matrix of products in cohomology is

\medskip\centerline{\begin{tabular}{r|rrrrrrrrrrrr}
 & $\lambda_{41}$ & $\lambda_{42}$ & $\lambda_{14}$ & $\lambda_{24}$ & $\lambda_{34}$ & $\lambda_{32}$ & $\delta_{12}$ & $\delta_{13}$ & $\delta_{23}$ & $\upsilon_{12}$ & $\upsilon_{13}$ & $\upsilon_{23}$ \\ \hline \rule{0mm}{4mm}
$\lambda_{41}$ & & & & & & & & & & & & $-g$ \\ \rule{0mm}{4mm}
$\lambda_{42}$ & & & & & & & $g$ & $-g$ & $g$ & $g$ & & $g$ \\ \rule{0mm}{4mm}
$\lambda_{14}$ & & & & & & & & & $-g$ & & & \\ \rule{0mm}{4mm}
$\lambda_{24}$ & & & & & & & & $g$ & & & & \\ \rule{0mm}{4mm}
$\lambda_{34}$ & & & & & & & & $-g$ & $g$ & $g$ & $-g$ & $g$ \\ \rule{0mm}{4mm}
$\lambda_{32}$ & & & & & & & $-g$ & $g$ & $-g$ & $-g$ & $g$ & $-g$\\  \rule{0mm}{4mm}
$\delta_{12}$ & & $-g$ & & & & $g$ & & & & & & \\ \rule{0mm}{4mm}
$\delta_{13}$ & & $g$ & & $-g$ & $g$ & $-g$ & & & & & & \\ \rule{0mm}{4mm}
$\delta_{23}$ & & $-g$ & $g$ & & $-g$ & $g$ & & & & & & \\ \rule{0mm}{4mm}
$\upsilon_{12}$ & & $-g$ & & & $-g$ & $g$ & & & & & & \\ \rule{0mm}{4mm}
$\upsilon_{13}$ & & & & & $g$ & $-g$ & & & & & & \\ \rule{0mm}{4mm}
$\upsilon_{23}$ & $g$ & $-g$ & & & $-g$ & $g$ & & & & & & \end{tabular}}

\medskip\noindent
where $g$ stands for the generator of $H^2(\C(|K_m|,2))$, zeros are not shown, and brackets for cohomology classes are omitted. In particular, replacing $\lambda_{42}$ by $\lambda_{42}':=\lambda_{42}+\lambda_{41}+\lambda_{34}+\lambda_{32}+\lambda_{24}+\lambda_{14}$, $\lambda_{34}$ by $\lambda_{34}':=\lambda_{34}+\lambda_{32}$ and $\lambda_{32}$ by $\lambda_{32}':=\lambda_{32}+\lambda_{42}$,
we get a cohomology basis whose only\footnote{Up to anticommutativity.} non-trivial products are
\begin{equation}\label{agradable}
\lambda'_{42}\smile \upsilon_{12}=\lambda_{24}\smile \delta_{13}=\lambda'_{32}\smile \upsilon_{13}=g\qquad\mbox{and}\qquad 
\lambda_{14}\smile \delta_{23}=\lambda_{41}\smile \upsilon_{23}=\lambda'_{34}\smile \delta_{12}=-g.
\end{equation}
This reflects the well known fact that $\C(|K_5|,2)$ is, up to homotopy, a closed orientable surface of genus~6.
}\end{example}

The cohomology ring $H^*(\C(|K_m,2|))$ becomes richer as $m$ increases (with $\C(|K_m|,2)$ no longer being a homotopy closed surface). Yet, some aspects of the particularly simple structure in~(\ref{agradable}) are kept for all $m>5$. Explicitly, let $\upsilon_{a,c}$, $\delta_{b,d}$ and $\lambda_{e,f}$ stand for the basis elements of $H^1(\C(|K_m,2|))$ represented, respectively, by the 1-cocycles $\md acmm$, $\md mmbd$ and $\ma ef$ described in Corollary~\ref{basedim1}.

\begin{corollary}\label{preservados}
Any cup product of the form $\hspace{.3mm}\delta_{b_1,d_1}\hspace{-.6mm}\cdot\delta_{b_2,d_2}$, $\hspace{.1mm}\upsilon_{a_1,c_1}\hspace{-.8mm}\cdot\hspace{-.2mm}\upsilon_{a_2,c_2}$ or $\hspace{.4mm}\lambda_{e_1,f_1}\hspace{-.6mm}\cdot\hspace{-.2mm}\lambda_{e_2,f_2}$ vanishes. On the other hand, a cup product $\delta_{b,d}\cdot\upsilon_{a,c}$ is nonzero if and only if $\{a,b\}\cap\{c,d\}=\varnothing$, in which case $\delta_{b,d}\cdot\upsilon_{a,c}$ is represented by $\mt aacbdd$.
\end{corollary}
\begin{proof}
This is a straightforward calculation using Proposition~\ref{fiarriba}. We only indicate the two main checking steps for the reader's benefit. In what follows we assume $m\geq6$. First, congruence \emph{(i)} in Proposition~\ref{fiarriba} is used to check that, modulo 1-faces not taking part on nonzero products~(\ref{pncero}), $\overline{\Phi}\ma ef$ is congruent to:
\begin{itemize}
\item $\displaystyle\sum_{\mbox{\tiny$i{<}m$}}\md iifm\!, \mbox{ for $e=m-1$ and $1\leq f\leq m-4$;}$

\vspace{-.5mm}
\item $\displaystyle\sum_{\mbox{\tiny$j{<}m$}}\left( \md eejm
+\displaystyle\sum_{\mbox{\tiny$x{<}m$}}\md exjj
-\displaystyle\sum_{\mbox{\tiny$y{<}m$}}\md yejj
\right)\!, \mbox{ for $1\leq e\leq m-3$ and $f=m-1$;}$

\vspace{.5mm}
\item $\md {m-1}{m-1}{m-3}m\hspace{.8mm}
-\hspace{-1.5mm}\displaystyle\sum_{\mbox{\tiny$\genfrac{}{}{0pt}{1}{i{<}m{-}1{>}j}{j{\neq}m{-}3}$}}\md iijm \hspace{.8mm}
-\hspace{-1.5mm}\displaystyle\sum_{\mbox{\tiny$i{<}m{-}1{>}j$}}\md i{m-1}jj\!, \mbox{ for $(e,f)=(m-1,m-3)$;}$

\vspace{-2.5mm}
\item $\displaystyle\md {m-2}{m-2}{m-1}m+
\md {m-2}{m-1}{m-1}{m-1}-
\sum_{\mbox{\tiny$y$}}\md y{m-2}{m-1}{m-1}-\hspace{-2mm}
\sum_{\mbox{\tiny$\genfrac{}{}{0pt}{1}{i{<}m{-}1{>}j}{i{\neq}m-2}$}}\!\!\left( \!\md iijm +\!\!\!
\sum_{\mbox{\tiny$\genfrac{}{}{0pt}{1}{x=m-\varepsilon}{\varepsilon\in\{1,2\}}$}} \md ixjj \!\right)\!,
\mbox{ for $(e,f)=(m-2,m-1)$;}$

\item $\displaystyle\sum_{\mbox{\tiny$i{<}m{-}1{>}j$}}\!\!\left(\md iijm+\md i{m-1}jj \right)\!, \mbox{ for $(e,f)=(m-2,m-3)$.}$
\end{itemize}

The above congruences together with those in items \emph{(ii)} and \emph{(iii)} of Proposition~\ref{fiarriba} are then used to check that each of the products asserted to vanish do so because there is no room for nonzero products~(\ref{pncero}) in the corresponding portion $\overline{\Phi}(x)\smile\overline{\Phi}(y)$ of~(\ref{cupprod}). Such an assertion is easily seen for products $\hspace{.3mm}\delta_{b_1,d_1}\hspace{-.6mm}\cdot\delta_{b_2,d_2}$ and $\hspace{.1mm}\upsilon_{a_1,c_1}\hspace{-.8mm}\cdot\hspace{-.2mm}\upsilon_{a_2,c_2}$, but the explicit details are not so direct for $\delta_{b,d}\cdot\upsilon_{a,c}$ and $\hspace{.4mm}\lambda_{e_1,f_1}\hspace{-.6mm}\cdot\hspace{-.2mm}\lambda_{e_2,f_2}$. In fact, in the latter two cases, a convenient order of factors needs to be chosen in order to ensure the vanishing of the corresponding $\overline{\Phi}(x)\smile\overline{\Phi}(y)$. See the table below. The order chosen is immaterial for the trivial-product conclusion, as cohomology cup products are anticommutative. Keep in mind that $H^*(\C(|K_m|,2))$ is torsion free, so that cup squares of 1-dimensional classes are trivial for free.

\begin{center}\begin{tabular}{|l|l|} 
\hline \rule{0mm}{4mm}  

$\delta_{b,d}\cdot\upsilon_{a,c}$, $a\in\{b,d\}$ &
$\upsilon_{a,c}\cdot\delta_{b,d}$, $c\in\{b,d\}$ 
\\ [2pt] \hline \hline \rule{0mm}{4mm}

$\lambda_{m-1,f_1}\cdot\lambda_{m-1,f_2}$, $\hspace{2.3mm}1\leq f_i\leq m-4$ &
$\lambda_{m-1,f}\cdot\lambda_{e,m-1}$, \hspace{4.4mm}$1\leq f\leq m-4$ \;and \, $1\leq e\leq m-3$
\\ [2pt] \hline \rule{0mm}{4mm}

$\lambda_{m-1,f}\cdot\lambda_{m-1,m-3}$, $1\leq f\leq m-4$ &
\\ \rule{0mm}{4mm}

$\lambda_{m-1,f}\cdot\lambda_{m-2,m-3}$, $1\leq f\leq m-4$ &
\\ \rule{0mm}{4mm}

$\lambda_{m-1,f}\cdot\lambda_{m-2,m-1}$, $1\leq f\leq m-4$ &
\\ [2pt] \hline \rule{0mm}{4mm}

$\lambda_{m-1,m-3}\cdot\lambda_{m-2,m-3}$ &
$\lambda_{m-1,m-3}\cdot\lambda_{e,m-1}$, $1\leq e\leq m-3$
\\ \rule{0mm}{4mm}

$\lambda_{m-1,m-3}\cdot\lambda_{m-2,m-1}$ & 
\\ [2pt] \hline \hline \rule{0mm}{4mm}

 & $\lambda_{e_1,m-1}\cdot\lambda_{e_2,m-1}$, \hspace{2mm}$1\leq e_i\leq m-3$ with $e_1>e_2$ 
\\ [2pt] \hline \rule{0mm}{4mm}

$\lambda_{m-2,m-3}\cdot\lambda_{m-2,m-1}$ &
$\lambda_{m-2,m-3}\cdot\lambda_{e,m-1}$, $1\leq e\leq m-3$ 
\\ \rule{0mm}{4mm}

& \raisebox{1mm}{$\lambda_{m-2,m-1}\cdot\lambda_{e,m-1}$, $1\leq e\leq m-3$} \rule{0mm}{5mm} \\ \hline
\end{tabular}\end{center}

Lastly, the fact that $\delta_{b,d}\cdot\upsilon_{a,c}$ is represented by $\mt aacbdd$ when $\{a,b\}\cap\{c,d\}=\varnothing$ follows by noticing that $\underline{\Phi}\left(\overline{\Phi}\md mmbd\smile \overline{\Phi}\md acmm\right)=\mt aacbdd$. Here $\mt aacbdd$ fails to represent one of our basis elements when either $(a,b)=(1,2)$ or $(a,b,c)=(1,3,2)$ or $(a,b,d)=(2,1,3)$ or $(a,b,c,d)=(2,1,3,4)$ (recall~(\ref{typ1})--(\ref{typ4})). In each such case, one of the relations $(E_1)$--$(E_6)$ in Corollary~\ref{bd2} applies to write (the cohomology class of) $\mt aacbdd$ in terms of basis elements. Either way, inspection of the relations $(E_1)$--$(E_6)$ shows that $\mt aacbdd$ represents a nonzero cohomology class.
\end{proof}

\section{Topological complexity}\label{TCsec}
Fix a positive integer $s\geq2$ and a path-connected space $X$. The $s$-th topological complexity $\TC_s(X)$ of $X$ is the sectional category of the evaluation map $e_s\colon PX\to X^s$ which sends a (free) path on $X$, $\gamma\in PX$, to
$$
e_s(\gamma)=\left(\gamma\left(\frac{0}{s-1}\right),\gamma\left(\frac{1}{s-1}\right),\ldots, \gamma\left(\frac{s-1}{s-1}\right)\right).
$$
The term ``sectional category'' is used in the reduced sense, so that $\TC_s(X)+1$ stands for the smallest number of open sets covering $X^s$ on each of which $e_s$ admits a section. For instance, the (reduced) Lusternik-Schnirelmann category $\cat(X)$ of $X$ is the sectional category of the evaluation map $e_1\colon P_0X\to X$ sending a based path $\gamma\in P_0X$ (i.e., $\gamma(0)=\star$, for a fixed base point $\star\in X$) to $e_1(\gamma)=\gamma(1)$. 

\begin{proposition}[{\cite[Theorem~3.9]{bgrt}}]\label{ulbTCnrararar}
For a $c$-connected space $X$ having the homotopy type of a CW complex, 
$$\cl(X)\leq\cat(X)\leq \hdim(X)/(c+1)
\quad\mbox{and}\quad
\zcl_s(X)\leq\TC_s(X)\leq s\cdot \cat(X).$$
\end{proposition}

Here $\hdim(X)$ denotes the minimal dimension of cell complexes homotopy equivalent to $X$, while $\cl(X)$ and $\zcl_s(X)$ stand, respectively, for the cup-length of $X$ and the $s$-th zero-divisor cup-length of $X$. Explicitly, $\cl(X)$ is the largest integer $\ell\geq0$ such that there are classes\footnote{For the purposes of this section, cohomology will be taken with mod 2 coefficients.} $c_j\in \widetilde{H}^*(X)$, $1\leq j\leq \ell$, with nonzero cup product. Likewise, $\zcl_s(X)$ is the largest integer~$\ell\geq0$ such that there are classes $z_j\in H^*(X^s)$, $1\leq j\leq\ell$ (``zero divisors'') with nonzero cup product and so that each factor restricts trivially under the diagonal $X \hookrightarrow X^s$.

Let $\Gamma$ be a 1-dimensional cell complex ---a graph. While the fundamental group of $\C(\Gamma,n)$ is a central character in geometric group theory, the topological complexity of $\C(\Gamma,n)$ becomes relevant for the task of planning collision-free motion of $n$ autonomous distinguishable agents moving on a $\Gamma$-shaped system of tracks. It is known that $\hdim(\C(\Gamma,n))$ is bounded from above by $m=m(\Gamma)$, the number of essential vertices of $\Gamma$ (see for instance~\cite[Theorem~4.4]{MR2171804}). Thus, Proposition~\ref{ulbTCnrararar} yields
\begin{equation}\label{uperestimate}
\TC_s(\C(\Gamma,n))\leq s\cdot m.
\end{equation}

For $s=2$, Farber proved in~\cite{farbergraphs} that~(\ref{uperestimate}) is an equality when $\Gamma$ is a tree and $n\geq 2m$, with the single (and well-known) exception of $(n,m)=(2,1)$ with the (unique) essential vertex of~$\Gamma$ having valency~$3$ ---we call it the ``$Y_2$-exception''. Farber also conjectured that the tree restriction would be superfluous in order to have equality in~(\ref{uperestimate}). The conjecture has recently been confirmed in~\cite{ben} by Knudsen, who proved equality in~(\ref{uperestimate}) for any $s\geq 2$ and any graph~$\Gamma$, as long as the ``stable'' restriction $n\geq 2m$ is kept (and the $Y_2$-exception is avoided). Note that the stable condition forces $\hdim(\C(\Gamma,n))=m$. More generally, it would be interesting to characterize the triples $(s,\Gamma,n)$ for which the (in principle) improved bound
\begin{equation}\label{uperestimateimproved}
\TC_s(\C(\Gamma,n))\leq s\cdot\hdim(\C(\Gamma,n))
\end{equation}
holds as an equality, preferably determining the value of $\hdim(\C(\Gamma,n))$. For instance, it is known from~\cite[Section~5]{MR4356250} that, for any~$s$ and $n$ (possibly with $n<2m$),
$$
\hdim(\C(\Gamma,n))=\cat(\C(\Gamma,n))=\min\{\lfloor n/2\rfloor,m\}
$$
when $\Gamma$ is a tree, in which case~(\ref{uperestimateimproved}) is an equality ---recalling the $Y_2$-exception. The goal of this section, Theorem~\ref{TCapp} below, is to add a new and completely different family of instances where equality holds in~(\ref{uperestimateimproved}) outside the stable regime.

\begin{theorem}\label{TCapp}
For $m\geq4$ and $s\geq2$,
$$
\TC_s(\C(|K_m|,2))=s\cdot\hdim(\C(|K_m|,2))=
\begin{cases} s, & \mbox{if $m=4$}; \\ 2s, & \mbox{otherwise.} \end{cases}
$$
\end{theorem}

Note that $\C(|K_m|,2)$ is empty for $m=1$, and disconnected for $m=2$, while $\C(|K_3|,2)$ leads to the $Y_2$-exception. On the other hand, the cases $m=4$ and $m=5$ in Theorem~\ref{TCapp} are well known in view of Corollary~\ref{caso4} and the last assertion in Example~\ref{sigma6}. We prove Theorem~\ref{TCapp} for $m\geq6$ by constructing~$2s$ zero divisors in $\C(|K_m|,2)$ with a nonzero cup product, and using Proposition~\ref{ulbTCnrararar} together with the obvious fact that $\hdim(\C(|K_m|,2))\leq2$. It is natural to think that the expected richness of cup products in general graph configuration spaces might lead to many more instances where~(\ref{uperestimateimproved}) would hold as an equality.

\smallskip
For integers $1\leq i\leq s\geq2$ and a cohomology class $x$ in a space $X$, consider the exterior tensor product $x_{(i)}:=1\otimes\cdots\otimes1\otimes x\otimes1\otimes\cdots\otimes1\in H^*(X)^{\otimes s}=H^*(X^s)$,
where the tensor factor $x$ appears on the $i$th position. The following result is straightforward to check:

\begin{lemma}\label{productosdeemilio}
Let $x,y,z,w$ be four elements in the mod 2 cohomology of a space $X$ satisfying the relations $x^2=y^2=xz=yz=yw=0$, then
$$
\left(\,\!\prod_{i=2}^s\left(x_{(1)}+x_{(i)}\right)\!\!\right)\!\!
\left(\,\!\prod_{i=2}^s\left(y_{(1)}+y_{(i)}\right)\!\!\right)\!\!
\left(z_{(1)}+z_{(s)}\right)\!\!
\left(w_{(1)}+w_{(s)}\right)=
zw\otimes xy\otimes xy\otimes\cdots\otimes xy+xy\otimes xy\otimes\cdots\otimes xy\otimes zw.
$$
\end{lemma}

\begin{proof}[Proof of Theorem~\ref{TCapp} for $m\geq6$]
In view of Corollary~\ref{preservados} and Lemma~\ref{productosdeemilio}, the 1-dimensional basis elements $x=\delta_{1,2},\,y=\upsilon_{3,4},\,z=\upsilon_{1,3},\,w=\delta_{2,4}\in H^*\C(|K_m|,2)$ yield a product of $2s$ zero divisors. The product is represented by
\begin{equation}\label{prodnocero}
\mt 113244\otimes\mt 334122\otimes\mt 334122\otimes\cdots\otimes\mt 334122+\mt 334122\otimes\mt 334122\otimes\cdots\otimes\mt 334122\otimes\mt 113244.
\end{equation}
The tensor factor $\mt 334122$ represents one of the basis elements in the previous section. However, as indicated in Figure~\ref{esesi}, we need to apply relation ($E_2$) in Corollary~\ref{bd2} in order to write the (cohomology class of the) tensor factor $\mt 113224$ as a sum $\sum b_i$ of basis elements $b_i$ (recall we work mod 2). If $m\geq6$, the basis element $\mt 335244$ appears as a summand~$b_i$, from which the non-triviality of~(\ref{prodnocero}) follows.
\end{proof}


{\sc \ 

Departamento de Matem\'aticas

Centro de Investigaci\'on y de Estudios Avanzados del I.P.N.

Av.~Instituto Polit\'ecnico Nacional n\'umero~2508, San Pedro Zacatenco

M\'exico City 07000, M\'exico.}

\tt jgonzalez@math.cinvestav.mx

jesus@math.cinvestav.mx

\end{document}